\documentclass[11pt,fleqn]{article}

\usepackage{amsfonts}
\usepackage{amssymb}
\usepackage{amsthm}
\usepackage[leqno]{amsmath}

\makeatletter
\newcommand{\leqnomode}{\tagsleft@true}
\newcommand{\reqnomode}{\tagsleft@false}
\makeatother

\usepackage[UKenglish]{babel}
\usepackage{enumerate}
\usepackage{footnote}
\usepackage{hyperref}
\usepackage{latexsym}
\usepackage{lmodern}
\theoremstyle{plain}
\newtheorem{theorem}{Theorem}
\newtheorem{conjecture}[theorem]{Conjecture}

\newtheorem{lemma}[theorem]{Lemma}
\theoremstyle{remark}

\usepackage{caption}
\DeclareCaptionType{algorithm}
\newcommand{\sset}[1]{\left\{#1\right\}}
\newcommand{\iskfour}{\textnormal{ISK}_4}

\newcommand{\isk}{\textnormal{ISK}_4}
\newcommand{\iskt}{\left\{\textnormal{ISK}_4, \textnormal{triangle}\right\}}
\newcommand{\isktk}{\left\{\textnormal{ISK}_4, \textnormal{triangle}, K_{3,3}\right\}}
\newcommand{\iskk}{\left\{\textnormal{ISK}_4, K_{3,3}\right\}}

\usepackage{marvosym}

\def\longbox#1{\parbox{0.85\textwidth}{#1}}

\title{Triangle-free graphs that do not contain an induced subdivision of $K_4$ are 3-colorable}
\author{Maria Chudnovsky\thanks{Supported by NSF grant DMS-1550991 and US Army Research Office Grant W911NF-16-1-0404.}
\\
Princeton University, Princeton, NJ 08544
\\
\\
Chun-Hung Liu
\\
Princeton University, Princeton, NJ 08544
\\
\\
Oliver Schaudt
\\
Universit\"at zu K\"oln, K\"oln, Germany
\\
\\
Sophie Spirkl
\\
Princeton University, Princeton, NJ 08544
\\
\\
Nicolas Trotignon\thanks{Partially supported by ANR project Stint
    under reference ANR-13-BS02-0007 and by the LABEX MILYON
    (ANR-10-LABX-0070) of Universit\'e de Lyon, within the program
    ``Investissements d'Avenir'' (ANR-11-IDEX-0007) operated by the
    French National Research Agency (ANR).}
\\
CNRS, LIP, ENS de Lyon, INRIA, Universit\'e de Lyon, France
\\
\\
Kristina Vu\v{s}kovi\'c
\\
School of Computing, University of Leeds, Leeds LS2 9JT, UK}
\date{\today}

\begin{document}
\sloppy

\maketitle

\begin{abstract} 
We show that triangle-free graphs that do not contain an induced subgraph isomorphic to a subdivision of $K_4$ are 3-colorable. This proves a conjecture of Trotignon and Vu\v{s}kovi\'c \cite{trotignon}.
\end{abstract}

\section{Introduction}

All graphs in this paper are finite and simple.
Let $G$ be a graph. For a vertex $v \in V(G)$, we denote its set of neighbors by $N(v)$, and we let $N[v] = \sset{v} \cup N(v)$. For $X, Y \subseteq V(G)$, we say that $X$ is \emph{complete} to $Y$ if every vertex in $X$ is adjacent to every vertex in $Y$; $X$ is \emph{anticomplete} to $Y$ if every vertex in $X$ is non-adjacent to every vertex in $Y$. A vertex $v \in V(G)$ is \emph{complete} (\emph{anticomplete}) to $X \subseteq V(G)$ if $\sset{v}$ is complete (anticomplete) to $X$. A set $X \subseteq V(G)$ is a \emph{cutset} for $G$ if there is a partition $(X, Y, Z)$ of $V(G)$ with $Y, Z \neq \emptyset$ and $Y$ anticomplete to $Z$. The cutset $X$ is a \emph{clique cutset} if $X$ is a (possibly empty) clique in $G$. For a graph $H$, we say that $G$ \emph{contains} $H$ if $H$ is isomorphic to an induced subgraph of $G$, and otherwise, $G$ is $H$-\emph{free}. For a family $\mathcal{F}$ of graphs, we say that $G$ is $\mathcal{F}$-\emph{free} if $G$ is $F$-free for every graph $F \in \mathcal{F}$. 

For a graph $G$ and $X \subseteq V(G)$, $G|X$ denotes the induced subgraph of $G$ with vertex set $X$. For $X \subseteq V(G)$, we let $G \setminus X = G|(V(G) \setminus X)$ and for $x \in V(G)$, we let $G \setminus x = G|(V(G) \setminus \sset{x})$.  By a {\em path} in a graph we mean an induced path. Let $C$ be a cycle in $G$. The \emph{length} of $C$ is $|V(C)|$. The \emph{girth} of $G$ is the length of a shortest cycle, and is defined to be $\infty$ if $G$ has no cycle. A \emph{hole} in a graph is an induced cycle of length at least four. An $\isk$ is a graph that is isomorphic to a subdivision of $K_4$. 

In \cite{trotignon} two of us studied the structure of $\isk$-free graphs,
and proposed the following conjecture (and proved several special cases of it):

\begin{conjecture}
\label{chi3}
If $G$ is $\iskt$-free, then $\chi(G) \leq 3$.
\end{conjecture}

In~\cite{le}, Conjecture~\ref{chi3} was proved with 3 replaced by 4. 

The main result of the present paper is the proof of Conjecture~\ref{chi3}. In fact, we prove a stronger statement, from which Conjecture~\ref{chi3} easily follows:

\begin{theorem} 
\label{v2}
Let $G$ be an $\iskt$-free graph. Then either $G$ has a clique cutset, $G$ is complete bipartite, or $G$ has a vertex of degree at most two. 
\end{theorem}

For an induced subgraph $H$ of $G$ we write $v \in H$ to mean $v \in V(H)$.
We use the same convention if $H$ is a path or a hole. For a path $P=p_1-\ldots-p_k$ we call the set $V(P) \setminus \{p_1,p_k\}$ the {\em interior} of $P$, 
and denote it by $P^*$.

A \emph{wheel} in a graph is a pair $W=(C,x)$ where $C$ is a hole and $x$ has at least 
three neighbors in $V(C)$. We call $C$ the {\em rim} of the wheel, and $x$ the
{\em center}. The neighbors of $x$ in $V(C)$ are called the {\em spokes} of 
$W$. Maximal paths of $C$ that do not contain any spokes in their interior 
are called the {\em sectors} of $W$. We write $V(W)$ to mean $V(C) \cup \{x\}$.

A graph is \emph{series-parallel} if it does not contain a subdivision of $K_4$ as a (not necessarily induced) subgraph.  
\begin{theorem}[\cite{duffin}] \label{thm:duffin}
Let $G$ be a series-parallel graph. Then $G$ is $\isk$-free, wheel-free, and $K_{3,3}$-free, and $G$ contains a vertex of degree at most two. 
\end{theorem}

The following two facts were proved in \cite{leveque}:

\begin{theorem}[\cite{leveque}] \label{lem:sub} Let $G$ be an $\iskt$-free graph. Then either $G$ is series-parallel, or $G$ contains a $K_{3,3}$ subgraph, or $G$ contains a wheel. If $G$ contains a subdivision of $K_{3,3}$ as an induced subgraph, then $G$ contains a $K_{3,3}$. 
\end{theorem}

\begin{theorem}[\cite{leveque}] \label{thm:todo}
If $G$ is an $\iskt$-free graph and $G$ contains $K_{3,3}$, then either $G$ is complete bipartite, or $G$ has a clique cutset. 
\end{theorem}

Thus to prove Theorem~\ref{v2} we need to analyze $\isktk$-free graphs that 
contain wheels. This approach was already explored in \cite{trotignon}, but we 
were able to push it further, as follows.

A wheel $W=(C,x)$  is {\em proper} if for every $v \in V(G) \setminus V(W)$
\begin{itemize}
\item there is a sector $S$ of $W$ such 
that $N(v) \cap V(C) \subseteq V(S)$. 
\item If $v$ has at least three neighbors in $V(C)$, then $v$ is adjacent to 
$x$.
\end{itemize}
(Please note that this definition is different from the one in \cite{trotignon}.)
We prove:

\begin{theorem}
\label{wheelmain0}
Let $G$ be an $\isktk$-free graph, and let $x$ be the center of a proper wheel in $G$. If
$W=(C,x)$ is a proper wheel with a minimum number of spokes subject to having center $x$, then
\begin{enumerate}
\item every component of $V(G) \setminus N(x)$ contains the interior of at most 
one sector of $W$, and
\item for every $u \in N(x)$,
the component $D$ of $V(G) \setminus (N(x) \setminus \{u\})$ such that 
$u \in V(D)$ contains the interiors of at most two sectors of $W$, 
and if $S_1,S_2$ are sectors with $S_i^* \subseteq V(D)$ for $i=1,2$, then 
$V(S_1) \cap V(S_2) \neq \emptyset$.
\end{enumerate}
\end{theorem}

Using Theorem~\ref{wheelmain0}  we can prove a variant of a conjecture from 
\cite{trotignon} that we now explain. For a graph $G$ and $x, y \in V(G)$, we say that $(x,y)$ is a \emph{non-center pair} for $G$ if neither $x$ nor $y$ is the center of a proper wheel in $G$, and $x=y$ or $xy \in E(G)$. We prove:

\begin{theorem} \label{thm:main0} Let $G$ be an $\isktk$-free graph which is not series-parallel, and let $(x,y)$ be a non-center pair for $G$. Then 
some $v \in V(G) \setminus (N[x] \cup N[y])$ has degree at most two. 
\end{theorem}

Here is the outline of the proof; the full proof is given in Section~5. We assume that $G$ is a  counterexample to Theorem~\ref{thm:main0} with $|V(G)|$ minimum. Since $G$ is not series-parallel, it follows from Theorem~\ref{lem:sub} that $G$ contains a wheel, and we show in Lemma~\ref{proper} that $G$ contains a proper wheel. 
Let $s \in V(G)$ be the center of a proper wheel chosen as in Theorem~\ref{wheelmain0}, and let $C_1, \dots, C_k$ be the components of $G \setminus N[s]$. By Theorem~\ref{wheelmain0}, it follows that $k>1$. For each $i$, let $N_i$ be the set of
vertices of $N(s)$ with a neighbor in $V(C_i)$, and let 
$G_i=G|(V(C_i) \cup N_i \cup \{s\})$.  We analyze the structure of 
the graphs $G_i$ using the minimality of $|V(G)|$.
It turns out that at most one $G_i$ is not series-parallel, and that (by 
contracting $C_i$'s) there is at most one value of $i$ for which $|V(C_i)|>1$. Also, if $|V(C_i)| > 1$, then $\sset{x, y} \cap V(C_i) \neq \emptyset$. 
We may assume that $|V(C_i)| =1$ for all $i \in \{1, \dots, k-1\}$, and that $\sset{x, y} \cap V(C_k) \neq \emptyset$.
Now consider the bipartite graph $G'$, which (roughly speaking) is the graph obtained from $G \setminus \{s\}$ by
contracting $V(C_k) \cup N_k$ to a single vertex $z$ if $|V(C_k)|>1$. It turns out that $G'$ is $\iskk$-free and has girth at least $6$, while cycles that do not contain $z$ must be even longer. Now either there is an easy win, or we find a cycle in $G'$ that contains a long path $P$ of vertices all of degree two in $G'$ and with $V(P) \subseteq V(G) \setminus (N[x] \cup N[y])$. Further analysis shows that at least one of these vertices has degree two in $G$, and Theorem~\ref{thm:main0} follows. 

This paper is organized as follows. In Section~2 we prove Theorem~\ref{wheelmain0}. Section~3 contains technical tools that we need to deduce that the graph $G'$ described above has various useful properties. In Section~4 we develop techniques to produce a cycle with a long path of vertices of degree two. In Section~5 we put all of our knowledge together to prove Theorems~\ref{v2} and~\ref{thm:main0} and deduce 
Conjecture~\ref{chi3}.

Let us finish this section with an easy fact about $\isk$-free graphs.
Given a hole $C$ and a vertex $v \not \in C$, $v$ is {\em linked} to $C$
if there are three paths $P_1,P_2,P_3$ such that
\begin{itemize}
\item $P_1^* \cup P_2^* \cup P_3^* \cup \{v\}$ is disjoint from $C$;
\item each $P_i$ has one end  $v$ and the other end in $C$, and there are 
no other edges between $P_i$ and $C$;
\item for $i, j \in \sset{1,2,3}$ with $i \neq j$, $V(P_i) \cap V(P_j) = \sset{v}$;
\item if $x \in P_i$ is adjacent to $y \in P_j$ then either $v \in \{x,y\}$
or $\sset{x,y} \subseteq V(C)$; and 
\item if $v$ has a neighbor $c \in C$, then $c \in P_i$ for some $i$.
\end{itemize}

\begin{lemma}
\label{nolink}
If $G$ is $\isk$-free, then no vertex of $G$ can be linked to a hole.
\end{lemma}

\section{Wheels} 
\begin{lemma}
\label{proper}
Let $G$ be an $\iskt$-free graph that contains a wheel. Then there is a proper wheel in $G$.
\end{lemma}

\begin{proof}
Let $G$ be an $\iskt$-free graph.
Let $W=(C,x)$ be a wheel in $G$ with $|V(C)|$ minimum. We claim that $W$ is a 
proper wheel. Suppose $v \in V(G) \setminus V(W)$ violates the definition of a 
proper wheel. 

If $v$ has at least three neighbors in  the hole $x-S-x$ for some sector $S$ of $W$, then $(x-S-x,v)$ is a wheel with shorter rim than 
$W$, a contradiction. So $v$ has at most two neighbors in every sector of $W$
(and at most one if $v$ is adjacent to $x$).  Therefore there exist sectors
$S_1,S_2$ of $W$ such that $v$ has a neighbor in $V(S_1) \setminus V(S_2)$
and a neighbor in $V(S_2) \setminus V(S_1)$. Also by the minimality 
of $|V(C)|$, every  path of $C$ whose ends are in $N(v)$ and with interior 
disjoint from $N(v)$ contains at most two spokes of $W$, and
we can choose $S_1,S_2$ and for $i=1,2$, label the ends of $S_i$ as $a_i,b_i$ such that either $b_1=a_2$,
or $b_1,a_2$ are the ends of a third sector $S_3$ of $W$ and $v$ has no neighbor
in $S_3^*$. If possible, we choose $S_1, S_2$ such that $b_1 = a_2$. 
If $v$  has two neighbors in $S_1$, denote them $s,t$ 
such that $a_1,t,s,b_1$ are in order in $C$. If $v$ has a unique neighbor in $S_1$, denote it by $s$. Let $z$ be the neighbor of $v$ in $S_2$ closest to 
$a_2$.

Assume first that $v$ is non-adjacent to $x$.
Suppose  $b_1 \neq a_2$. By Lemma~\ref{nolink}, $x$ cannot be linked to the hole
$z-S_2-a_2-S_3-b_1-S_1-s-v-z$, and it follows that $z \neq b_2$.
If $v$ has two neighbors in $S_1$, then  $v$ can be 
linked to $x-S_3-x$ via the paths $v-s-S_1-b_1$, $v-t-S_1-a_1-x$, $v-z-S_2-a_2$; and if $v$ has a 
unique neighbor in $S_1$, then $s$ can be linked to $x-S_3-x$ via the paths
$s-S_1-b_1$, $s-S_1-a_1-x$, $s-v-z-S_2-a_2$ (note that by the choice of $S_1, S_2$ and since $b_1 \neq a_2$, it follows that $s \neq b_1$). In both cases, this is contrary to Lemma~\ref{nolink}.
This proves that $b_1=a_2$. Let $y$ be the neighbor of $v$ in $S_2$ closest 
to $b_2$.
Now if $v$ has two neighbors in $S_1$, then
$v$ can be linked to $x-S_1-x$ via the paths $v-s$, $v-t$, $v-y-S_2-b_2-x$,
contrary to Lemma~\ref{nolink}. So $v$ has a unique neighbor in $S_1$, and similarly a
unique neighbor in $S_2$. It follows that $s, b_1$ and $z$ are all distinct. Now we can link $x$ to $s-S_1-b_1-S_2-z-v-s$ via the
paths $x-b_1$, $x-a_1-S_1-s$, and $x-b_2-S_2-z$, contrary to Lemma~\ref{nolink}.

This proves that $v$ is adjacent to $x$, and so $v$ has at most one neighbor 
in every sector of $W$. If $b_1 \neq a_2$, then $v$ can be linked to 
$x-S_3-x$ via the paths  $v-s-S_1-b_1$, $v-x$, $v-z-S_2-a_2$, and if
$b_1=a_2$, then $s, b_1$ and $z$ are all distinct and hence $x$ can be linked to the hole $s-S_1-b_1-S_2-z-v-s$ via the
paths $x-b_1$, $x-v$, and $x-b_2-S_2-z$; in both cases contrary to Lemma~\ref{nolink}.
This proves that every $v \in V(G) \setminus V(W)$ satisfies the condition in the 
definition, and so $W$ is a proper wheel in $G$. 
\end{proof} 

Let $W = (C,v)$ be a wheel. We call $x$ \emph{proper} for $W$ if either $x \in V(C) \cup \sset{v}$; or  
\begin{itemize}
\item all neighbors of $x$ in $V(C)$ are in one sector of $W$; and 
\item if $x$ has more than two neighbors in $V(C)$, then $x$ is adjacent to $v$. 
\end{itemize}

A vertex $x$ is \emph{non-offensive} for a wheel $W=(C,v)$ if  there exist two sectors $S_1, S_2$ of $W$ such that
\begin{itemize}
\item $x$ is adjacent to $v$;  
\item $x$ has neighbors in $S_1$ and in $S_2$;  
\item $N(x) \cap V(C) \subseteq V(S_1) \cup V(S_2)$;  
\item $S_1$ and $S_2$ are consecutive; and 
\item if $u \in V(G) \setminus V(W)$ is adjacent to $x$, then $N(u) \cap V(C) \subseteq V(S_1) \cup V(S_2)$.  
\end{itemize}
If $V(S_1) \cap V(S_2)=\{a\}$, we also say that  $x$ is {\em $a$-non-offensive}.

\begin{lemma} \label{lem:many-nbrs}
Let $G$ be an $\iskt$-free graph. Let $W = (C,v)$ be a wheel in $G$. Let $S_1, S_2$ be consecutive sectors of $W$, and let $x \in N(v) \setminus V(W)$ be a vertex such that $N(x) \cap V(C) \subseteq V(S_1) \cup V(S_2)$ and 
$N(x) \cap V(S_1), N(x) \cap V(S_2) \neq \emptyset$. 
Then $N(x) \cap V(C) \subseteq S_1^* \cup S_2^*$ and for $\sset{i, j} = \sset{1,2}$, $|N(x) \cap (V(S_i) \setminus V(S_j))| \geq 3$. 
\end{lemma}
\begin{proof}
Since $x$ is adjacent to $v$ and $G$ is triangle-free, it follows that
$N(x) \cap V(C) \subseteq S_1^* \cup S_2^*$.
Suppose for a contradiction that $|N(x) \cap (V(S_1) \setminus V(S_2))| \leq 2$.If  $|N(x) \cap (V(S_1) \setminus V(S_2))| = 2$, then $x$ has exactly three neighbors in the hole $v - S_1 - v$, contrary to Lemma~\ref{nolink}. It follows that $|N(x) \cap (V(S_1) \setminus V(S_2))| = 1$. Let $z$ denote the neighbor of $x$ in $V(S_1)$.  Let $\sset{w} = V(S_1) \cap V(S_2)$, and let $y$ denote the neighbor of $x$ in $V(S_2)$ closest to $w$ along $S_2$. Then $x$ can be linked to the hole $v-S_1-v$ via the three paths $x-v$, $x-z$, and $x-y-S_2-w$. This is a contradiction to Lemma~\ref{nolink}, and the result follows. 
\end{proof}

We say that  wheel $W=(C,v)$ is \emph{$k$-almost proper} if there are spokes $x_1, \dots, x_k$ of $W$ and a set $X \subseteq V(G) \setminus V(W)$ such that
\begin{itemize}
\item no two spokes in $\sset{x_1, \dots, x_k}$ are consecutive;
\item $W$ is proper in $G\setminus X$;
\item for every $x$ in $X$ there exists $i$ such that $x$ is $x_i$-non-offensive.
\end{itemize}

\begin{lemma}\label{lem:one-almost}
Let $G$ be an $\iskt$-free graph, and let $W=(C,v)$ be a 1-almost proper wheel in $G$. Let $x_1$ and $X$ be as in the definition of a 1-almost proper wheel, and let $S_1$ and $S_2$ be the sectors of $W$ containing $x_1$.

Then there exists a proper wheel $W'$ in $G$ with center $v$ and the same number of spokes as $W$. Moreover, either $W=W'$, or $V(W') \setminus V(W) = \sset{x^*}$ where $x^*$ is a non-offensive vertex for $W$, and $V(W) \setminus V(W') \subseteq V(S_1^*) \cup V(S_2^*) \cup \sset{x_1}$. 
\end{lemma}
\begin{proof}
 We may assume that $X \neq \emptyset$, for otherwise $W$ is proper in $G$. For $x \in X$, let $P(x)$ denote the longest path in $G|(V(S_1) \cup V(S_2))$ starting and ending in a neighbor of $x$. Let $x^* \in X$ be a vertex with $|V(P(x^*))|$ maximum among vertices in $X$, and let $Y$ denote interior of $P(x^*)$. Let $C' = G|((V(C) \cup \sset{x^*}) \setminus Y)$. It follows that $W' = (C', v)$ is a wheel. 
Moreover, $N(v) \cap V(C') = ((N(v) \cap V(C)) \setminus \sset{x_1}) \cup \sset{x^*}$, and therefore $W'$ has the same number of spokes as $W$. 

If $W'$ is proper, the result follows. Therefore, we may assume that there is a vertex $y \in V(G) \setminus V(W')$ that is not proper for $W'$. Let $S_1', S_2'$ denote the sectors of $W'$ containing $S_1 \setminus Y$ and $S_2 \setminus Y$, respectively. 

 Suppose first that $y \in V(W)$, and consequently $y \in Y$. Since $x^*$ has at least two neighbors in each of $S_1$ and $S_2$ by Lemma~\ref{lem:many-nbrs}, it follows that $|V(P(x^*))| \geq 4$. Consequently, either $N(y) \cap V(C') \subseteq S_1'$ or $N(y) \cap V(C') \subseteq S_2'$. Moreover, $N(y) \cap V(C') \subseteq N[x]$, and therefore $|N(y) \cap V(C')| \leq 1$. This implies that $y$ is proper for $W$, a contradiction. This proves that $y\not \in V(W)$. 

Next, we suppose that $y \in X$. It follows that $N(y) \cap V(C') \subseteq V(S_1') \cup V(S_2')$. Since $y$ is not proper for $W'$, but $y$ is adjacent to $v$, it follows that $N(y) \cap V(S_1'), N(y) \cap V(S_2') \neq \emptyset$. By Lemma~\ref{lem:many-nbrs}, $y$ has at least two neighbors in $S_1'$. But then $V(P(x^*)) \subsetneq V(P(y))$, a contradiction to the choice of $x^*$. This proves that $y \in V(G) \setminus (X \cup V(W))$. 

If $y \not\in N(x^*)$, then $N(y) \cap V(C') \subseteq N(y) \cap V(C)$, and since $y \not\in X$, it follows that $y$ is proper for $W$ and thus $y$ is proper for $W'$. Consequently $y \in N(x^*)$. Since $x^*$ is non-offensive for $W$, it follows that $N(y) \cap V(C) \subseteq V(S_1) \cup V(S_2)$. Since $y$ is proper for $W$, we may assume by symmetry that $N(y) \cap V(C) \subseteq V(S_1)$. It follows that $N(y) \cap V(C') \subseteq V(S_1')$. Since $y$ is adjacent to $x^*$ and $G$ is triangle-free, it follows that $y$ is non-adjacent to $v$, and hence $|N(y) \cap V(C)| \leq 2$. This implies that $|N(y) \cap V(C')| \leq 3$, and by Lemma~\ref{nolink}, it is impossible for $y$ to have exactly three neighbors in $C'$ since $G$ is $\isk$-free. Therefore, $|N(y) \cap V(C')| \leq 2$, and therefore $y$ is proper for $W'$. This is a contradiction, and it follows that $W'$ is proper in $G$, and hence $W'$ is the desired wheel. 
\end{proof}

\begin{lemma} \label{lem:two-almost} 
Let $G$ be an $\iskt$-free graph, and let $W=(C,v)$ be a 2-almost proper wheel in $G$. Then there exists a proper wheel in $G$ with center $v$ and at most the same number of spokes as $W$. 
\end{lemma}
\begin{proof}
Let $x_1, x_2$ and $X$ be as in the definition of a 2-almost proper wheel, and let $S_1, S_2$ be the sectors of $W$ containing $x_1$. Let $X_1$ denote the set 
$x_1$-non-offensive vertices in $X$,
and let $X_2 = X \setminus X_1$. We may assume that $X_1,X_2$ are both 
non-empty, for otherwise the result follows from Lemma~\ref{lem:one-almost}.
It follows that $W$ is 1-almost proper, but not proper,  in $G\setminus X_2$. 
Let $W'$, $x^*$ be as in Lemma~\ref{lem:one-almost}. So $W'$ is a proper wheel in $G \setminus X_2$. If $W'$ is 1-almost proper in $G$, then the result of the Lemma follows from Lemma~\ref{lem:one-almost}. So we may assume that $W'$ is not 1-almost proper in $G$. Since every vertex of
$V(G) \setminus X_2$ is proper for $W'$, we deduce that some vertex of
$x \in X_2$ is not proper and not $x_2$-non-offensive for $W'$.

By the definition of $X_1$, and since $W$ is 2-almost proper in $G$, it follows that $N(x) \cap V(C)$ is contained in the sectors $S_3, S_4$ of $W$ containing $x_2$. Since $x_1$ and $x_2$ are not consecutive, $S_3, S_4 \not\in \sset{S_1, S_2}$, and so by Lemma~\ref{lem:one-almost}, $V(S_3) \cup V(S_4) \subseteq V(W')$. Consequently,  $S_3$ and $S_4$ are sectors of $W'$. Since $G$ is triangle-free and every vertex in $X$ is adjacent to $v$, it follows that $x$ is not adjacent to $x^*$. Therefore, $N(x) \cap V(C') \subseteq V(S_3) \cup V(S_4)$, 
$x$ has both a neighbor in $S_3$ and a neighbor in $S_4$, and $S_3, S_4$ are 
the sectors of $W'$ containing $x_2$.  Let $s_3$ denote the neighbor of $x$ in $S_3$ furthest from $x_2$, and let $s_4$ denote the neighbors of $x$ in $S_4$ 
furthest from $x_2$. We may assume that among all vertices of $X_2$ that are
not $x_2$-non-offensive for $W'$, $x$ is chosen so that the path of  $C$ from $s_3$ to 
$s_4$ containing $x_2$ is maximal.

Since $x$ is not $x_2$-non-offensive for $W'$, 
there exists a vertex $u \in N(x) \setminus V(W')$ with a neighbor in $V(C') \setminus (V(S_3) \cup V(S_4))$. Since $x$ is $x_2$-non-offensive for $W$, it follows that $u$ has a neighbor in $V(C') \setminus V(C) = \sset{x^*}$, and so $u$ is adjacent to $x$ and $x^*$. 

Since $x$ and $x^*$ are non-offensive for $W$, it follows that $N(u) \cap V(C) \subseteq (V(S_1) \cup V(S_2)) \cap (V(S_3) \cup V(S_4))$.
Since $G$ is triangle-free, $u$ is non-adjacent to $v$, and therefore 
$u \not \in X$. Consequently, $u$ is proper for $W$,  and all the neighbors of 
$u$ in $C$ belong to one sector of $W$. It follows that $u$ has at most one neighbor in $V(C)$. Suppose that  $u$ has exactly one neighbor in $V(C)$. Then $u$ has three neighbors in the cycle arising from $C'$ by replacing $s_3 - S_3 - x_2 - S_4 - s_4$ by $s_3 - x - s_4$, contrary to Lemma~\ref{nolink}. It follows that $u$ has no neighbors in $V(C)$. 

Let $P_1'$ denote the path of $C'$ from $s_3$ to $x^*$ not containing 
$x_2$, and let $P_1$ be $x-s_3-P_1'-x^*$. Let $P_2'$ denote the path
of $C'$ from $s_4$ to $x^*$ not containing $x_2$, and let $P_2$ be 
$x-s_4-P_1'-x^*$. Let $D = G|(V(P_1) \cup \sset{u})$. Since $x_1$ and $x_2$ are not consecutive, each of $P_1^*, P_2^*$ contains at least one neighbor of $v$, and so $W'' = (D, v)$ is a wheel with fewer spokes than $W$. Let $S_3'$ denote the sector of $W''$ containing $x$ but not containing $u$. If $W''$ is proper in $G$, then the result follows. Therefore, we may assume that there is a vertex $y\in V(G) \setminus V(W'')$ that is not proper for $W''$. 

Since every vertex in $V(W') \setminus V(W'')$ and every vertex in 
$V(W) \setminus V(W'')$ has at most one neighbor in $V(W'')$, it follows that $y \not\in V(W) \cup V(W')$. Suppose that $y\not\in N(u)$. If $y \in X_2$, then $N(y) \cap V(D) \subseteq (V(S_3) \cup \sset{x}) \cap V(D)$, and so $y$ is proper for $W''$, since $y$ is adjacent to $v$, a contradiction. Thus $y \not \in X_2$, and so $y$ is proper for $W'$. 
If $y \not\in N(x)$, then $N(y) \cap V(D) \subseteq N(y) \cap V(C')$, and again $y$ is proper for $W''$, a contradiction. Thus  $y\in N(x)$, but since $x$ is non-offensive for $W$, $N(y) \cap V(C) \subseteq V(S_3) \cup V(S_4)$, and so $N(y) \cap V(D) \subseteq V(S_3')$. Since $y$ is not proper for $W''$, $y$ is non-adjacent to $v$ and has at least three neighbors in $S_3'$. But $y$ is proper for $W'$, and so $y$ has at most two neighbors in $S_3$; thus $y$ has exactly three neighbors in $S_3'$ and hence in $D$ contrary to Lemma~\ref{nolink}. This contradiction implies that $y \in N(u)$. 

Since $y$ is not proper for $W''$, it follows that $y$ has a neighbor in $P_1^*$, and since $G$ is triangle-free, it follows that $y$ is non-adjacent to $x, x^*$. We claim that $y$ has no neighbor in $P_2$. Suppose that it does. If $y \in X_2$, then, since $y$ is adjacent to $u$ and has a neighbor in $P_1^*$, we deduce that $y$ is not 
$x_2$-non-offensive for $W'$, and 
the claim follows from the maximality of the path of $C$ from $s_3$ to $s_4$ 
containing $x_2$. 
Thus we may assume that $y \not \in X_2$. Consequently,  $y$ is proper for 
$W'$, a contradiction. This proves the claim. 

Let $z_1$ be the neighbor of $y$ in $V(P_1)$ closest to $x$ along $P_1$, and let $z_2$ be the neighbor of $y$ in $V(P_1)$ closest to $x^*$ along $P_1$. Let 
$D'$ be the hole $x^*-P_2-x-u-x^*$. If $z_1 \neq z_2$, we can link $y$ to 
$D'$ via the paths $y-z_1-P_1-x$, $y-z_2-P_1-x^*$ and $y-u$, and if
$z_1=z_2$, then we can link $z_1$ to $D'$ via the paths $z_1-P_1-x$, 
$z_1-P_1-x^*$ and and $z_1-y-u$, in both cases contrary to Lemma~\ref{nolink}.
This proves Lemma~\ref{lem:two-almost}.
\end{proof}

Throughout the remainder of this section $G$ is an $\isktk$-free graph, and $W=(C,x)$ is a proper wheel in $G$ with minimum number of spokes (subject to having center $x$). 

\begin{lemma}
\label{unique}
Let $P=p_1- \ldots -p_k$ be a path such that $p_1,p_k$ have neighbors in $V(C)$,
$V(P) \subseteq V(G) \setminus V(W)$, and there are no edges between $P^*$ and $V(C)$.  Assume that no sector of $W$ contains
$(N(p_1) \cup N(p_k)) \cap V(C)$.  
For $i \in \{1,k\}$, if $x$ is 
non-adjacent to $p_i$, then  $p_i$ has a unique 
neighbor in $C$.
\end{lemma}

\begin{proof}
Let  $S_1,S_2$ be distinct sectors of $W$ 
such that $N(p_1) \cap V(C) \subseteq V(S_1)$,
and $N(p_k) \cap V(C) \subseteq V(S_2)$.
We may assume that $p_1$ is non-adjacent to $x$, and so $p_1$ has at most
two neighbors in $C$.
Since $p_1$ cannot be linked to the hole $C$ (or the hole obtained from 
$C$
by rerouting $S_2$ through $p_k$) via two one-edge paths and $P$, it follows that
$p_1$ has a unique neighbor in $C$.
\end{proof}

\begin{theorem}
\label{paths}
Let $P=p_1-\ldots -p_k$ be a path with 
$V(P) \subseteq V(G) \setminus V(W)$ such that $x$ has at most one
neighbor in $P$. 
\begin{enumerate}
\item If $P$ contains no neighbor of $x$, then there is a sector $S$ of
$W$ such that every edge from $P$ to $C$ has an end in $V(S)$.
\item If $P$ contains exactly one neighbor of $x$, then there are two
sectors $S_1,S_2$ of $W$ such that $V(S_1) \cap V(S_2) \neq \emptyset$, and
 every edge from $P$ to $C$ has an end in $V(S_1) \cup V(S_2)$ (where possibly $S_1=S_2$).
\end{enumerate}
\end{theorem}

\begin{proof}
Let $P$ be a path violating the assertions of the theorem and assume that $P$ is chosen 
with $k$ minimum. Since $W$ is proper, it follows that $k>1$.
Our first goal is to show that $x$ has a neighbor in $V(P)$.

Suppose that $x$ is anticomplete to $V(P)$. 
Then, by the minimality of $k$,  
there exist two sectors $S_1,S_2$ of $W$ such that every edge
 from $\{p_1, \dots, p_{k-1}\}$ to $V(C)$  has an end in $V(S_1)$,
and every  edge
 from $\{p_2, \dots, p_k\}$ to $V(C)$  has an end in $V(S_2)$.
It follows that $S_1 \neq S_2$.
Then $p_1$ has a neighbor 
in $V(S_1) \setminus V(S_2)$, and $p_k$ has a neighbor in 
$V(S_2) \setminus V(S_1)$,
and every edge from $P^*$ to $V(C)$ has an end in $V(S_1) \cap V(S_2)$.
For $i=1,2$ let $a_i,b_i$  be the ends of $S_i$. 
We may assume  that $a_1,b_1,a_2,b_2$ appear in $C$ in 
this order and that $a_1 \neq b_2$.  
Let $Q_1$ be the path of $C$ from $b_2$ to $a_1$ not using $b_1$, and let 
$Q_2$ be the path of $C$ from $b_1$ to $a_2$ not using $a_1$. We can choose 
$S_1$, $S_2$ with $|V(Q_2)|$ minimum (without changing $P$). 
Let $s$ be the neighbor of $p_1$ in $S_1$ closest to $a_1$,
$t$ the neighbor of $p_1$ in $S_1$ closest to $b_1$, $y$ the neighbor
of $p_k$ in $S_2$ closest to $a_2$ and $z$ the neighbor of $p_k$ in
$S_2$ closest to $b_2$. Then $s \neq b_1$
and $z \neq a_2$. It follows that $V(Q_2) \cap \{s,z\} =\emptyset$. 
Moreover, if $V(S_1) \cap V(S_2) \neq \emptyset$,
then $b_1=a_2$ and $V(Q_2)=\{b_1\}$, and in all cases 
$V(Q_1)$ is anticomplete to $P^*$.

Now $D_1=s-p_1-P-p_k-z-S_2-b_2-Q_1-a_1-S_1-s$ is a hole.

\vspace*{-0.4cm}\begin{equation} \label{W_1}
  \longbox{\emph{$W_1=(D_1,x)$ is a wheel with fewer spokes than $W$.}}
\end{equation}

Since $V(Q_2) \cap  V(D_1)=\emptyset$ and $V(Q_2)$ contains 
a neighbor of $x$, it follows that  $x$ has fewer neighbors in $D_1$ than it
does in $C$. It now suffices to 
show that $x$ has at least three neighbors in $Q_1$. Since $a_1,b_2 \in V(Q_1)$, we may assume that $x$ has no neighbor in $Q_1^*$, and 
$Q_1$ is a sector of $W$. Since not every edge between $V(P)$ and $V(C)$ has an 
end in $V(Q_1)$, it follows that $t \neq a_1$ or $y \neq b_2$. 
By symmetry, we may assume that $t \neq a_1$. Since $x$ cannot be linked
to $W$ by Lemma~\ref{nolink}, it follows that $x$ has at least four neighbors in $V(C)$, and therefore  $V(S_1) \cap V(S_2) = \emptyset$. Consequently, 
$P^*$ is anticomplete to $V(C)$. It follows from Lemma~\ref{unique} that
$s=t$.
Now  we can link $s$ to the hole $a_1-Q_1-b_2-x-a_1$ via the paths
$s-S_1-a_1$, $s-p_1-P-p_k-z-S_2-b_2$ and $s-S_1-b_1-x$, 
contrary to Lemma~\ref{nolink}. 
This proves  that $W_1=(D_1,x)$ is a wheel with fewer spokes than $W$. 
This proves (\ref{W_1}).

\vspace*{0.4cm}

By the choice of $W$, it follows from (\ref{W_1}) that $W_1$ is not proper.
Let $S_0$ be the sector $a_1-S_1-s-p_1-P-p_k-z-S_2-b_2$ of $(D_1,x)$.
Since $W$ is a proper wheel and $W_1$ is not a proper wheel, we have that:

\vspace*{-0.4cm}\begin{equation} \label{casesnonbr}
  \longbox{\emph{There exists $v \in  V(G) \setminus V(W_1)$ such that either
\begin{itemize}
\item $v$ is non-adjacent to $x$, and $v$ has at least three neighbors in $S_0$
and $N(v) \cap V(D_1) \subseteq V(S_0)$, or
\item there is a sector $S_3$ of $W$ with
$V(S_3) \subseteq V(Q_1)$, such that $v$ has a neighbor in 
$V(S_3) \setminus V(S_0)$ and a neighbor in $V(S_0) \setminus V(S_3)$,
and $N(v) \cap V(C) \subseteq V(S_3)$.
\end{itemize}}}
\end{equation}

Let $v$ be as in (\ref{casesnonbr}).
First we show that  $v \not \in V(C)$.  The only vertices of $C$ that may have more than one neighbor in $D_1$ are $b_1$ and $a_2$, and that only happens if 
$b_1=a_2$. 
But $N(b_1) \cap V(D_1) \subseteq V(S_0)$ and $b_1$ is adjacent to $x$, so
$b_1$ does not satisfy the conditions described in the bullets. Thus $v \not \in V(C)$.

\vspace*{-0.4cm}\begin{equation} \label{uniquenonbr}
  \longbox{\emph{$v$ has a unique neighbor in $P$.}}
\end{equation}

If the first case of (\ref{casesnonbr}) holds, then the statement of (\ref{uniquenonbr}) 
follows immediately from the minimality of $k$
(since $v$ is non-adjacent to $x$), and
so we may assume that the second case of (\ref{casesnonbr}) holds. 
Observe that no vertex of $V(Q_1)$ is contained both in a sector with end
$a_1$ and in a sector with end $b_2$, and therefore we may assume that $v$
has a neighbor in a sector that does not have end $b_2$.
If $v$ is non-adjacent to $x$, we get a contradiction to the minimality of $k$.
So we may assume that $v$ is adjacent to $x$, and therefore $v$ has a neighbor 
in $S_3^*$, and $b_2 \not \in V(S_3)$. Let $S_4$ be the sector of $D_1$ 
such that $b_2 \in V(S_4)$ and $V(S_4) \subseteq V(Q_1)$. 
Suppose that $y \neq b_2$ or $V(S_3) \cap V(S_4) = \emptyset$. Let $i \in \{1, \dots, k\}$  be maximum such that $v$ is adjacent to $p_i$. Now the path
$v-p_i-P-p_k$ violates the assertions of the theorem, and  so it follows from the 
minimality of $k$ that $N(v) \cap V(P) \subseteq \{p_1,p_2\}$. Therefore,  
since $G$ is  triangle-free, it follows that $v$ has a unique neighbor in $P$,
and (\ref{uniquenonbr}) holds. So we may assume that 
$y=b_2$ and there exists $a_3 \in V(C)$ such 
that $V(S_4) \cap V(S_3)=\{a_3\}$. 
Let $R$ be the path from
$v$ to $a_3$ with $R^* \subseteq S_3^*$. Now we can link $v$ to $x-S_4-x$
via the paths $v-x$, $v-R-a_3$ and $v-p_i-P-p_k-b_2$, where $i$ is maximum such that $v$ is adjacent to $p_i$, contrary to Lemma~\ref{nolink}. 
This proves (\ref{uniquenonbr}).

\vspace*{0.4cm}

In view of (\ref{uniquenonbr})
let $N(v) \cap V(P)=\{p_j\}$. In the case of the first bullet of (\ref{casesnonbr}),
since $v$ cannot be linked to the hole $x-S_0-x$ by Lemma~\ref{nolink}, it follows that $v$ has at least four neighbors in $S_0$, and therefore at least three neighbors
in $V(S_1) \cup V(S_2)$, contrary to the fact that $W$ is proper.
So the case of the second bullet  of (\ref{casesnonbr}) holds. Since $W$ is 
proper
$N(v) \cap (V(S_0) \setminus V(S_3)) \subseteq V(P)$, and 
$N(v) \cap V(D_1) \subseteq V(S_0) \cup V(S_3)$.

\vspace*{-0.4cm}\begin{equation} \label{edges}
  \longbox{\emph{There are  edges between $P^*$ and $V(C)$.}}
\end{equation}

Suppose not.
By Lemma~\ref{unique},  $s=t$ and $y=z$.
We claim that in this case $b_1 \neq a_2$, for if $b_1 = a_2$, then 
$b_1$ can be linked to the hole $x-a_1-S_1-s-p_1-P-p_k-z-S_2-b_2-x$
via the paths $b_1-x$, $b_1-S_1-s$ and $b_1-S_2-z$, contrary to Lemma~\ref{nolink}.
If $v$ has a unique neighbor $r$ in $C$, then
$p_j$ can be linked to $C$ via the paths $p_j-P-p_1-s$, $p_j-P-p_k-z$ and 
$p_j-v-r$, contrary to Lemma~\ref{nolink}, so $v$ has at least two  neighbors in
$C$. Recall that
$N(v) \cap V(C) \subseteq V(S_3)$. Let $D$ be the hole obtained from $C$ by 
rerouting $S_3$ through $v$.
Then $s,z \in V(D)$, and  $p_j$ can be linked to $D$ via the paths
$p_j-P-p_1-s$, $p_j-P-p_k-z$ and 
$p_j-v$, contrary to Lemma~\ref{nolink}. This proves (\ref{edges}).

\vspace*{0.4cm}

If follows from (\ref{edges}) that
$b_1=a_2$ and $b_1$ has neighbors in $P^*$. 
Now, by considering the path from a neighbor of $b_1$ in $P^*$  to $v$ with interior in $P^*$ if $v$ has a neighbor in $P^*$, and the paths $v-p_1$ or $v-p_k$ 
if $v$ has no neighbor in $P^*$,  
the minimality of $k$ implies that $v$ is  adjacent to $x$ and one of 
$a_1,b_2$ belongs to $S_3$.

By symmetry we may assume $a_1 \in V(S_3)$.
Let $R$ be the path from $v$ to 
$a_1$ with $R^* \subseteq V(S_3)$.  
Now $x$ can be linked to the hole $v-R-a_1-S_1-s-p_1-P-p_j-v$ via the paths 
$x-v$,  $x-a_1$ and $x-b_2-S_2-z-p_k-P-p_j$, contrary to Lemma~\ref{nolink}.

\vspace*{0.4cm}

In summary, we have now proved:

\vspace*{-0.4cm}\begin{equation} \label{xnbr}
  \longbox{\emph{If $P'$ is a path violating the assertion of the theorem and
$|V(P')|=k$, then  
 $x$ has a neighbor in $V(P')$.}}
\end{equation}

By \eqref{xnbr}, $x$ has a neighbor in $V(P)$, say $x$ is adjacent to $p_i$.
Then $p_i$ is the unique neighbor of $x$ in $V(P)$.
By the minimality of $k$,  
there exist two distinct sectors $S_1,S_2$ of $W$ such that 
$p_1$ has a neighbor in $V(S_1) \setminus V(S_2)$, and $p_k$ has a neighbor in
$V(S_2) \setminus V(S_1)$. By~\eqref{xnbr}, if $1<i<j$, then
every edge from $\{p_1, \dots, p_{i-1}\}$ to $V(C)$  has an end in $V(S_1)$,
and every  edge from $\{p_{i+1}, \dots, p_k\}$ to $V(C)$  has an end in 
$V(S_2)$; if $i=1$ then every edge from $V(P) \setminus \{p_1\}$ to $V(C)$
has and end in $V(S_2)$; and if $i=k$ then  every edge from 
$V(P) \setminus \{p_k\}$ to $V(C)$ has and end in $V(S_1)$.

For $j=1,2$, let $a_j,b_j$  be the ends of $S_j$.

\vspace*{-0.4cm}\begin{equation} \label{chords}
  \longbox{\emph{One of the following statements holds:
\begin{itemize}
\item there are no edges between $V(C)$ and $P^*$, or
\item we can choose $S_1,S_2$ such that   $a_1,b_1,a_2,b_2$ appear in $C$ in 
order and  there is a sector $S_3$ with ends $b_1,a_2$, and every edge between
$V(C)$ and $P^*$ is from $b_1$ to $\{p_2, \dots, p_{i-1} \}$
or from $p_i$ to $S_3^*$, or from $a_2 $ to $\{p_{i+1}, \dots, p_{k-1}\}$.
\end{itemize}}}
\end{equation}

Suppose (\ref{chords}) is false. It follows that there are edges between $P^*$ and 
$V(C)$. Since $G$ is triangle-free, $p_i$ is anticomplete to $N(x) \cap V(C)$.
Suppose that 
there is  sector $S_3$ of $W$ and an edge from $S_3^*$ to $P^*$. By the 
minimality of
$k$ we deduce that $S_3 \not \in \{S_1,S_2\}$, $1<i<k$ and $p_i$ has a 
neighbor in $S_3^*$. Again
by the minimality of $k$ it follows that there exist
sectors $S_1'$, $S_2'$ such that $V(S_j') \cap V(S_3) \neq \emptyset$
for $j=1,2$ and every edge from $\{p_1, \dots, p_{i-1}\}$ to $C$ has an end
in $S_1'$, and every edge from $\{p_{i+1}, \dots, p_k\}$ to $C$ has an end in
$S_2'$. Now we can choose $S_1=S_1'$ and $S_2=S_2'$. We may assume that
$a_1,b_1,a_2,b_2$ appear in $C$ in this order, and so $b_1$ and $a_2$ are the 
ends of $S_3$. Since $p_i$ has a neighbor in $S_3^*$, the minimality of $k$
implies that $\{p_2, \dots, p_i\}$ is anticomplete to $V(S_1) \setminus \{b_1\}$, and $\{p_i, \dots, p_{k-1}\}$ is anticomplete to $V(S_2) \setminus \{a_2\}$,
and the second bullet is satisfied. So $P^*$ is anticomplete to 
$V(C) \setminus N(x)$. Since there are edges between $P^*$ and $V(C)$, and since
$p_i$ is anticomplete to $N(x) \cap V(C)$, by symmetry we may assume that
there is an edge between $\{p_2, \dots, p_{i-1}\}$
and $t \in N(x) \cap V(C)$. Then $t \in V(S_1)$. Let $S_3$ be the other 
sector of $W$ incident with $t$. By the minimality of $k$ it follows
that $S_2$ can be chosen so that $V(S_3) \cap V(S_2) \neq \emptyset$, and 
again the case of the second bullet holds. This proves (\ref{chords}).

\vspace*{0.4cm}
If the second bullet of (\ref{chords}) holds, 
let $Q_1$ be the path of $C$ from $b_2$ to $a_1$ not using $b_1$, 
and let $Q_2=S_3$. To define $Q_1$ and $Q_2$, let us now assume that  the case of the first bullet holds.
We may assume that $a_1,b_1,a_2,b_2$  appear in $C$ in this order. Also,
$a_1,b_1,a_2,b_2$ are all distinct, since $P$ violates the 
assertion of the theorem.
Let $Q_1$ be the path of $C$ from $b_2$ to $a_1$ not using $b_1$, and let 
$Q_2$ be the path of $C$ from $b_1$ to $a_2$ not using $a_1$. We may assume 
that $S_1$, $S_2$ are chosen with $|V(Q_2)|$ minimum (without changing $P$).

Since $W$ is proper, it follows that $N(p_1) \cap V(C) \subseteq V(S_1)$ 
and $N(p_k) \cap V(C) \subseteq V(S_2)$. 
Let $s$ be the neighbor of $p_1$ in $S_1$ closest to $a_1$,
$t$ the neighbor of $p_1$ in $S_1$ closest to $b_1$, $y$ the neighbor
of $p_k$ in $S_2$ closest to $a_2$ and $z$ the neighbor of $p_k$ in
$S_2$ closest to $b_2$. Then $s \neq b_1$ and $z \neq a_2$.

Let $D_1$ be the hole $a_1-S_1-s-p_1-P-p_k-z-S_2-b_2-Q_1-a_1$. Then $W_1=(D_1,x)$
is a wheel with fewer spokes than $W$.
We may assume that (subject to the minimality of $k$) $P$ was chosen
so that $V(Q_1)$ is (inclusion-wise) minimal. By Lemma~\ref{nolink}, $x$
has a neighbor in $V(D_1) \setminus \{a_1,b_1,p_i\}$, and so $x$ has a
neighbor in $Q_1^*$.

Let $S_0$ be the sector $a_1-S_1-s-p_1-P-p_i$, and let $T_0$ be the
sector $p_i-P-p_k-z-b_2$ of $(D_1,x)$.

\vspace*{-0.4cm}\begin{equation} \label{S_0T_0}
  \longbox{\emph{No vertex $v \in V(G) \setminus V(W_1)$ has both a neighbor
in $V(S_0) \setminus V(T_0)$ and a neighbor in $V(T_0) \setminus V(S_0)$.}}
\end{equation}

Suppose (\ref{S_0T_0}) is false, and let $v \in V(G)\setminus V(W_1)$ 
be such that $v$ has a neighbor  in $V(S_0) \setminus V(T_0)$ and a neighbor 
in $V(T_0) \setminus V(S_0)$. 

First we claim that $v$ is adjacent to $x$. 
Suppose $v$ has a neighbor in $V(a_1-S_1-s)$. Since $W$ is proper and 
$a_1,s \not \in V(S_2)$ (because $P$ violates the statement of the theorem), 
it follows that $v$ has no 
neighbor in $V(z-S_2-b_2)$.  Consequently $v$ has a neighbor in
$V(T_0) \setminus (V(S_2) \cup V(S_0))$.
Let $j$ be maximum such that $v$ is adjacent to $p_j$, then $j>i$.
Now applying \eqref{xnbr} to the path $v-p_j-P-p_k$ we deduce that
$v$ is adjacent to $x$, as required. Thus we may assume that 
$N(v) \cap (V(S_0) \cup V(T_0)) \subseteq V(P)$. Let $j$ be minimum
and $l$ maximum such that $v$ is adjacent to $p_j$, $p_l$. Then 
$j<i$ and $l>i$.  Applying \eqref{xnbr} to the path $p_1-P-p_j-v-p_l-P-p_k$,
we again deduce that $x$ is adjacent to $v$. This proves the claim.

In view of the claim,  Lemma~\ref{lem:many-nbrs} implies that $v$ has at least 
two neighbors in  $V(T_0) \setminus V(S_0)$ and at least two neighbors in 
$V(S_0) \setminus V(T_0)$. But now, rerouting $P$ through $v$ (as in the previous paragraph), we get a contradiction to the minimality of $k$. This 
proves \eqref{S_0T_0}.

\vspace*{0.4cm}

\vspace*{-0.4cm}\begin{equation} \label{non-offensive}
  \longbox{\emph{Every non-offensive vertex for $W_1$ is 
either $a_1$-non-offensive or $b_2$-non-offensive.}}
\end{equation}

Let $v$ be a non-offensive vertex for $W_1$. Since $W$ is proper, it follows 
that $N(v) \cap V(C)$ is included in a unique sector of $W$.  Consequently,
$v$ is either $a_1$-non-offensive, or $b_2$-non-offensive,
or $p_i$ non-offensive. However, (\ref{S_0T_0}) implies that $v$ is not
$p_i$-non-offensive, and (\ref{non-offensive}) follows.

\vspace*{0.4cm}

Let $X$ be the set of all non-offensive vertices for $W_1$. It follows from
Lemma~\ref{lem:two-almost} that $W_1$ is not proper in $V(G) \setminus X$.

\vspace*{-0.4cm}\begin{equation} \label{cases}
  \longbox{\emph {There exists $v \in V(G) \setminus (V(W_1) \cup X)$ such that 
one of the following holds:
\begin{itemize}
\item $v$ is non-adjacent to $x$, and $v$ has at least three neighbors in $S_0$, and $N(v) \cap V(D_1) \subseteq V(S_0)$.
\item $v$ is non-adjacent to $x$, and $v$ has at least three neighbors in $T_0$, and $N(v) \cap V(D_1) \subseteq V(T_0)$.
\item $v$ has a neighbor in $V(S_0) \setminus V(T_0)$ and a neighbor in 
$V(T_0) \setminus V(S_0)$, and $N(v) \cap V(D_1) \subseteq V(S_0) \cup V(T_0)$.
\item (possibly with the roles of $S_0$ and $T_0$ exchanged) 
there is a sector  $S_4$ of $W$ with $V(S_4) \subseteq V(Q_1)$ 
 such that $v$ has a neighbor in  $V(S_4) \setminus (V(S_0) \cup V(T_0))$, 
$v$ has a neighbor in  $V(S_0) \setminus V(S_4)$,
$v$ does not have a neighbor in $V(T_0) \setminus (V(S_0) \cup V(S_4))$, 
and $N(v) \cap V(C) \subseteq V(S_4)$.
\end{itemize}}}
\end{equation}

We may assume that the first three bullets of (\ref{cases}) do not hold.  
Since $W$ is proper and $W_1$ is not, (possibly switching the roles of $S_0$ 
and $T_0$) there exists $v \in V(G) \setminus V(W_1)$ and a sector  $S_4$ of 
$W$ with  $V(S_4) \subseteq V(Q_1)$, such 
that $v$ has a neighbor in  $V(S_4) \setminus V(S_0)$, $v$ has a neighbor in 
$V(S_0) \setminus V(S_4)$, and $N(v) \cap V(C) \subseteq V(S_4)$. 
But now (\ref{S_0T_0}) implies that the last 
bullet  of (\ref{cases}) holds. This proves (\ref{cases}).

\vspace*{0.4cm}

Let $v \in V(G)$ be as in (\ref{cases}). Next we show that:

\vspace*{-0.4cm}\begin{equation} \label{uniquenbr}
  \longbox{\emph{$v$ has a unique neighbor  in $V(P)$.}}
\end{equation}

Suppose that $v$ has at least two  neighbors in $P$.
In the first two cases of (\ref{cases}) we get a contradiction to the minimality of $k$. The third case is impossible by~(\ref{S_0T_0}). 
Thus we may assume that the
case of the fourth bullet of (\ref{cases})  holds.
We may assume that $N(v) \cap V(P) \subseteq V(S_0)$, and in particular $v$
has a  neighbor in $\{p_1, \dots, p_{i-1}\}$. Suppose first that 
$v$ is non-adjacent to $x$. Since $v$ has a neighbor in 
$V(S_4) \setminus V(S_0)$,  the minimality of $k$ implies that
$t=a_1$ and $a_1 \in V(S_4)$, and also that $b_2 \in V(S_4)$, 
contrary to the fact that $x$ has a neighbor in $Q_1^*$.  
So $v$ is adjacent to $x$, and therefore $v$ has a neighbor in $S_0^*$.

Since $W$ is proper, $N(v) \cap (V(S_0) \setminus  V(S_4)) \subseteq V(P)$.
Let $Q$ be the path from $v$ to $p_1$ with $Q^* \subseteq V(P)$.
Suppose first that $a_1 \not \in V(S_4)$. Let $S_5$ be the sector of $W$ 
with end $a_1$ and such that $V(S_5) \subseteq V(Q_1)$, and let $b_3$ be the 
other end of $S_5$. Since $Q$ is shorter than $P$, it follows from the 
minimality of $k$ that
$V(S_4) \cap V(S_5) = \{b_3\}$ and $t=a_1$. Let $R$ be the path from $v$ to 
$b_3$ with $R^* \subseteq S_4^*$. Then $x$ has exactly three neighbors in the
hole $v-R-b_3-S_5-a_1-p_1-Q-v$, contrary to Lemma~\ref{nolink}. This proves that
$a_1 \in V(S_4)$. 

Let $b_3$ be the other end of $S_4$, let $S_5$ be 
the second sector of $W$ incident with $b_3$, and let $a_3$ be the other 
end of $S_5$. 
Since $v \not \in X$, it follows that $v$ has a neighbor 
$u \in V(G) \setminus V(W_1)$ such that $u$ has a neighbor in 
$V(D_1) \setminus (V(S_4) \cup V(S_0))$. Since $G$ is triangle-free,
$u$ is non-adjacent to $x$.

Suppose first that $u$ has a neighbor in $V(Q_1) \setminus V(S_4)$. Since $G$ is
 triangle-free  and $v$ has at least  two neighbors in $V(P)$, it follows that 
$i \geq 4$, and therefore $k \geq 4$.  Consequently,  the path $u-v$ is 
shorter than $P$,
and so it follows from the minimality of $k$ that 
$N(u) \cap V(C) \subseteq V(S_5)$. Let $R$ be the path from $v$ to $b_3$ with
$R^* \subseteq S_4^*$, and let $D_2$ be the hole $v-R-b_3-x-v$. Let $p$ be the neighbor of $u$ in $V(S_5)$ closest to $b_3$, and let $q$ be the neighbor of
$u$ in $V(S_5)$ closest to $a_3$. If $p \neq q$, we can link $u$ to $D_2$ via the paths $u-p-S_5-b_3$, $u-q-S_5-a_3-x$ and $u-v$, and if $p=q$ we can link
$p$ to $D_2$ via the paths $p-u-v$, $p-S_5-b_3$ and $p-S_5-a_3-x$, in both cases
contrary to Lemma~\ref{nolink}. This proves that $u$ has no neighbor in
$V(Q_1) \setminus V(S_4)$, and therefore $u$ has a neighbor in 
$V(T_0) \setminus V(S_0)$. 

Next we define a new path $Q$.
If $u$ has a neighbor in $V(T_0) \cap V(S_2)$, let $Q$ be the path $u-v$.
If $u$ is anticomplete to $V(T_0) \cap V(S_2)$, let $j$ be maximum such that $u$ is adjacent to $p_j$; then $j>i$; let $Q$ be the path $v-u-p_j-P-p_k$.
Since $i>4$, in both cases $|V(Q)|<k$ and $x$ has a unique neighbor in
$V(Q)$. It follows from the minimality of $k$ that $z=y=b_2=a_3$.  Since 
$P$ violates the theorem, it follows that $p_1$ has a neighbor in 
$V(S_1) \setminus \{a_1\}$.

Let $T$ be the path from $v$ to $a_1$ with 
$T^* \subseteq V(S_4)$. Suppose that $s \neq t$. Let
$D_3$ be the hole $x-a_1-S_1-s-p_1-t-S_1-b_1-x$. Now
$v$ can be linked to $D_3$ via the paths $v-x$, $v-Q-p_1$ 
(short-cutting through a neighbor of $b_1$ if possible) and 
$v-T-a_1$, contrary to Lemma~\ref{nolink}. Thus $s=t$, and therefore 
$s \neq a_1$. Bow we can link $v$ to $x-S_1-x$ via the paths $v-x$, $v-Q-p_1-s$ 
(short-cutting through a neighbor of $b_1$ if possible) and 
$v-T-a_1$, contrary to Lemma~\ref{nolink}. 
This proves~(\ref{uniquenbr}).

\vspace*{0.4cm}

In view of (\ref{uniquenbr}) let $p_j$ be the unique neighbor of $v$ in $V(P)$.

\vspace*{-0.4cm}\begin{equation} \label{case4}
  \longbox{\emph{The fourth case of (\ref{cases}) holds. }}
\end{equation}

Suppose first that the case of the first bullet of (\ref{cases}) happens. 
Then by Lemma~\ref{nolink} 
$v$ has at least four neighbors in the hole $x-S_0-x$, and so, in view of 
(\ref{uniquenbr}),  $x$ has
at least three neighbors in the path $a_1-S_1-s$, contrary to the fact that
$W$ is proper. By symmetry it follows that the cases of first two bullets 
of (\ref{cases}) do  not happen. Suppose that the
case of the third bullet of (\ref{cases}) happens. Since by (\ref{uniquenbr})
$v$ has a unique 
neighbor in 
$V(P)$, it follows that $v$ has a neighbor in $(V(S_0) \cup V(T_0)) \setminus V(P)$. By symmetry we may assume that $v$ has a neighbor in  $z-S_2-b_2$,
and, since $W$ is proper, $v$ is anticomplete to $V(S_0) \setminus V(P)$.
Consequently, $p_j \in V(S_0) \setminus V(T_0)$, and so $j<i$. By the 
minimality of $k$ (applied to the path $p_1-P-p_j-v$), it follows that $j=k-1$,
and therefore $i=k$.  Then  $\{v,p_k\}$ is anticomplete to 
$V(C) \setminus V(S_2)$, since $W$ is proper.
By (\ref{xnbr}) $v$ is adjacent to $x$. But now we get a contradiction to
Lemma~\ref{lem:many-nbrs} applied to $v$ and $W_1$.
This proves (\ref{case4}).

\vspace*{0.4cm}

In the next claim we further restrict the structure of $P$.

\vspace*{-0.4cm}\begin{equation} \label{structure}
  \longbox{\emph{One of the following statements holds:
\begin{itemize}
\item there are  edges between $P^*$ and $V(C)$, or
\item $j=1$ and we can choose $S_4$ so that $a_1 \in V(S_4)$, or
\item $j=k$ and we can choose $S_4$ so that $b_2 \in V(S_4)$.
\end{itemize}}}
\end{equation}

Suppose that (\ref{structure}) is false.
Assume first that $j \not \in \{1,k\}$. 
Then $p_j$  is anticomplete to $V(C)$, since by assumption, there are no edges between $P^*$ and $C$.
If $s=t$, $y=z$ and $v$ has a unique neighbor $r$ in 
$S_4$, then $r \in V(S_4) \setminus (V(S_1) \cup V(S_2))$, and 
$p_j$
can be linked to $C$ via the paths $p_j-P-p_k-z$, $p_j-v-r$ and 
$p_j-P-p_1-s$, contrary to Lemma~\ref{nolink}. If some of $p_1,p_k,v$
have several neighbors in $C$, then similar linkages work for the holes
obtained from $C$ by rerouting $S_1$ through $p_1$, $S_2$ through $p_k$, 
and $S_4$ through $v$, respectively. This proves that $j \in \{1,k\}$, 
and by symmetry may assume that $j=1$.   Then $S_4$ cannot be chosen so that 
$a_1 \not \in V(S_4)$, for otherwise (\ref{structure}) holds.
By the minimality of $k$ and by (\ref{xnbr}), since
$S_4$ cannot be chosen so that $a_1 \in V(S_4)$, 
it follows that $x$ is adjacent to one of $p_1,v$ and $k=2$.
Since $G$ is triangle-free, $x$ has exactly one neighbor in $\{p_1,v\}$.
Let $R$ be the path from $v$ to $a_1$ with $R^* \subseteq V(C) \setminus \{b_1\}$. Let $Q_1'$ be the subpath of $R$ from an end of $S_4$ to $a_1$. 
Then  $V(Q_1') \subseteq V(Q_1)$ and $b_2 \not \in V(Q_1')$,  and so the
path $p_1-v$  contradicts the choice of $P$. This proves (\ref{structure}).

\vspace*{0.4cm}

The goal of the next two claims is to obtain more information about $i$ and $j$.

\vspace*{-0.4cm}\begin{equation} \label{jisi}
  \longbox{\emph{$i=j$.}}
\end{equation}

Suppose not; by symmetry we may assume that $j<i$.
Suppose first that $x$ is non-adjacent to $v$. 
By (\ref{xnbr}) and the minimality of $k$, it follows that the first assertion of
the theorem holds for the path $p_1-P-p_j-v$; therefore
$a_1=t$ and $S_4$ can be chosen so that $a_1 \in V(S_4)$. Since $W$ is proper it follows that $v$ has at most two
neighbors in $S_4$. If $v$ has exactly two neighbors, then, in view of 
(\ref{chords}),  $v$ can be 
linked to $x-S_4-x$ via two one-edge paths and the path $v-p_j-P-p_i-x$,
contrary to Lemma~\ref{nolink}. Therefore $v$ has a unique neighbor $r$ in
$S_4$. Now, again in view of (\ref{chords}),  $p_j$ can be linked to $x-S_4-x$ via the paths $p_j-v-r$, 
$p_j-P-p_1-a_1$ and $p_j-P-p_i-x$, again contrary to Lemma~\ref{nolink}. This 
proves that  $v$ is  adjacent to $x$, and, since $G$ is triangle-free,
$v$ has a neighbor in $S_4^*$.  It follows that the choice of $S_4$ is unique.
Let $R$ be the path from 
$v$ to $a_1$ with $R^* \subseteq V(C) \setminus \{b_1\}$. 
Suppose $a_1 \in V(S_4)$. Then $R^* \subseteq S_4^*$.
In this case, because of (\ref{chords}) and since $b_1 \neq s$,   
$p_j$ can be linked to the hole
$v-R-a_1-x-v$ via the path $p_j-v$, $p_j-P-p_1-s-S_1-a_1$ and 
$p_j-P-p_i-x$, contrary to Lemma~\ref{nolink}. Thus $a_1 \not \in V(S_4)$.
Let $S_5$ be the sector of $W$ with end $a_1$ such that 
$V(S_5) \subseteq V(Q_1)$. If $t=a_1$ and $V(S_4) \cap V(S_5) \neq \emptyset$,
then $x$ has exactly three neighbors in the hole $v-R-a_1-p_1-P-p_j-v$, contrary to Lemma~\ref{nolink}. Therefore the path $p_1-P-p_j-v$ violates 
the assertion of the theorem, and so the minimality of $k$ implies that $j=k-1$
and consequently $i=k$. Then by (\ref{chords}) $a_2$ is 
anticomplete to  $V(P)  \setminus \{p_k\}$.  Since $j\neq k$ and $a_1 \not \in V(S_4)$ (the choice
of $S_4$ is now unique), it follows from (\ref{structure}) that there
are edges between $P^*$ and $V(C)$. Now by (\ref{chords})
there is a sector $S_3$ of $W$ with 
ends $a_2,b_1$,  and $b_1$ has a neighbor in $P^*$.
Then there is  a path 
$T$ from $b_1$ to $p_k$ with $T^* \subseteq P^*$, 
$b_1-S_3-a_2-S_2-y-p_k-T-b_1$ is a hole and $x$ has exactly three neighbors in
it, contrary to Lemma~\ref{nolink} (observe that $y \neq b_2$ because $G$ has
no triangles).   This proves (\ref{jisi}).

\vspace*{0.4cm}

Since $G$ is triangle-free, (\ref{jisi}) implies that 
$x$ is non-adjacent to $v$. 

\vspace*{-0.4cm}\begin{equation} \label{ends}
  \longbox{\emph{$i \leq 2$ and $i \geq k-1$.}}
\end{equation}

Suppose (\ref{ends}) is false. By symmetry we may assume that $k-i>1$. 
Consequently $k>2$.
Suppose   that $S_4$ can be chosen so that $a_1 \in V(S_4)$. If $v$ has a unique
neighbor $r$  in $V(S_4)$, then, since $s \neq b_1$,  
$p_i$ can be linked to $x-S_4-x$ via the
paths $p_i-v-r$, $p_i-x$ and $p_i-P-p_1-s-S_1-a_1$, a contradiction. Thus  $v$
has at least two neighbors in $V(S_4)$. Now, again using the fact that $s \neq b_1$,  $p_i$ can be linked to  the hole 
obtained from $x-S_4-x$ by rerouting $S_4$ through $v$ via the paths 
$p_i-v$, $p_i-x$ and $p_i-P-p_1-s-S_1-a_1$, again contrary to Lemma~\ref{nolink}.
Thus $S_4$ cannot be chosen so that $a_1 \in V(S_4)$. 
Let $S_5$ be the sector of $W$ with end $a_1$ such that $V(S_5) \subseteq V(Q_1)$. Since $i \leq k-2$, the minimality of $k$ applied to the path
$p_1-P-p_i-v$ implies that $t=a_1$ and 
$V(S_4) \cap V(S_5) \neq \emptyset$. In particular,
$i \neq 1$. It follows from (\ref{structure}) that there
are edges between $P^*$ and $V(C)$, and by (\ref{chords}) there
is a sector $S_3$ of $W$ with ends $b_1,a_2$ and every edge from 
$p_i$ to $V(C)$ have an end in $S_3^*$.
Together with the minimality of $k$ (using the path
$p_i-v$), this implies that  $p_i$ is anticomplete to $V(C)$. 
If $v$ has a unique neighbor $r$ in 
$S_4$ (and therefore $r \neq b_2$) and $p_k$ has a unique neighbor in $S_2$, then $p_i$ can be linked to
$C$ via the paths $p_i-v-r$, $p_i-P-p_1-a_1$ (short-cutting through neighbors of
$b_1$ if possible), and $p_i-P-p_k-z$ (short-cutting through neighbors of $a_2$
if possible). If $v$ has at least two neighbors in $V(S_4)$ or
or $p_k$ has at least two neighbors in $V(S_2)$, then the same linkage
works rerouting $S_4$ through $v$, and $S_2$ through $p_k$, respectively.
This proves (\ref{ends}).

\vspace*{0.4cm}

It follows from (\ref{jisi}) and (\ref{ends}) that either
\begin{itemize}
\item $k=3$ and $i=j=2$, or 
\item $k=2$.
\end{itemize}

If $k=3$ and $i=j=2$, then by (\ref{structure}) there are edges between
$P^*$ and $V(C)$,  and so by (\ref{chords}) there is 
a sector $S_3$ with ends $a_2,b_1$, so that
$p_2$ has neighbors in $S_3^*$; now the path $p_2-v$ contradicts the 
minimality of $k$. Thus $k=2$, and we may assume that $i=1$, by symmetry. 
Since $G$ is triangle-free, it follows that $p_1$ is non-adjacent to $a_1,b_1$.
Since now $P^*=\emptyset$ is  anticomplete to $V(C)$, it follows 
from (\ref{structure}) that we 
can choose $S_4$ with $a_1 \in V(S_4)$. Since $v$ is non-adjacent to $x$
and $W$ is proper, it follows that $v$ has
at most two neighbors in $S_4$. If $v$ has exactly two neighbors in $S_4$, 
then $v$ can be linked to the hole $x-S_4-x$ via two one-edge paths, and the
path $v-p_1-x$, contrary to Lemma~\ref{nolink}. Thus $v$ has a unique neighbor
$r$ in $V(S_4)$. Now $p_1$ can be linked to $x-S_4-x$ via the paths
$p_1-v-r$, $p_1-x$ and $p_1-s-S_1-a_1$, again contrary to Lemma~\ref{nolink}.
This completes the proof of Theorem~\ref{paths}.  
\end{proof}

We can now prove Theorem~\ref{wheelmain0} which we restate:

\begin{theorem}
\label{wheelmain}
Let $G$ be an $\isktk$-free graph, and let $x$ be the center of a proper wheel in $G$. If $W=(C,x)$ is a proper wheel with a minimum number of spokes subject to having center $x$, then

\begin{enumerate}
\item every component of $V(G) \setminus N(x)$ contains the interior of at most 
one sector of $W$, and
\item for every $u \in N(x)$,
the component $D$ of $V(G) \setminus (N(x) \setminus \{u\})$ such that 
$u \in V(D)$ contains the interiors of at most two sectors of $W$, 
and if $S_1,S_2$ are sectors with $S_i^* \subseteq V(D)$ for $i=1,2$, then 
$V(S_1) \cap V(S_2) \neq \emptyset$.
\end{enumerate}
\end{theorem}
 
\begin{proof}
To prove the first statement, we observe that if some  component of 
$V(G) \setminus N(x)$ contains the interiors of two sectors of $W$, then this 
component contains a path violating the first assertions of 
Theorem~\ref{paths}.
 
For the second statement, suppose $D$ contains the interiors of two disjoint 
sectors $S_1,S_2$ of
$W$. Since $|D \cap N(x)|=1$, we get a path in $D$ violating the second
assertion of Theorem~\ref{paths}.
This proves~Theorem~\ref{wheelmain}.  
\end{proof}

\section{Proper Wheel Centers}

In the proof of our main theorem, we perform manipulations on $\isktk$-free graphs; in this section, we show that this preserves being $\isktk$-free, and that no vertex becomes the center of a proper wheel. 

\begin{lemma} \label{lem:starcutset} Let $G$ be an $\isktk$-free graph, $s \in V(G)$, $K$ a component of $G \setminus N[s]$, and $N$ the set of vertices in $N(s)$ with a neighbor in $K$. Let $H = G|(V(K) \cup N \cup \sset{s})$. Then $s$ is not the center of a proper wheel in $H$, and for $v \in V(H) \setminus \sset{s}$, if $v$ is the center of a proper wheel in $H$, then $v$ is the center of a proper wheel in $G$.
\end{lemma}
\begin{proof}
Since $H \setminus N[s]$ is connected, it follows that $s$ is not the center of a proper wheel in $H$ by Theorem~\ref{wheelmain}. Let $v \in V(H) \setminus \sset{s}$ be the center of a proper wheel $W = (C, v)$ in $H$. For all $w \in V(G) \setminus V(H)$, $N(w) \cap V(C) \subseteq N[s]$, and since $G$ is triangle-free, it follows that every vertex $w \in V(G) \setminus V(H)$ either has at most one neighbor in $V(C)$, or $N(w) \cap V(C) \subseteq N(s)$. 

Suppose that $W$ is not proper in $G$. Then there exists a vertex $w$ such that either $w$ has more than two neighbors in a sector of $W$, but $w$ is not adjacent to $v$, or $w$ has neighbors in at least two sectors of $W$. It follows that $w$ has more than one neighbor in $V(C)$, and thus in $N(s)$. Suppose that $w$ has three distinct neighbors $a,b,c$ in $V(C) \cap N(s)$. Let $P$ be a shortest path connecting two of $a,b,c$, say $a$ and $b$, with interior in $K$; then $s$ is anticomplete to $P^*$. If $c$ is anticomplete to $P$, then $G|(V(P) \cup \sset{w, s, a, b, c})$ is an $\isk$. Otherwise, by the minimality of $|V(P)|$, $P^*$ consists of a single vertex $x$, and $\sset{w,s,x,a,b,c}$ induces a $K_{3,3}$ subgraph in $G$, a contradiction. So $w$ has exactly two neighbors $a$ and $b$ in $V(C)$, and thus $a$ and $b$ are in different sectors of $W$. Since $a, b \in N(s)$ and $W$ is proper in $H$, it follows that $s \in V(C)$ and $s$ is a spoke of $W$; let $S, S'$ be the two sectors of $W$ containing $s$. But then $v$ can be linked to the cycle $s-a-w-b-s$ via $v-s$ and the two paths with interiors in $S \setminus s$ and $S' \setminus s$. This is a contradiction by Lemma~\ref{nolink} and it follows that $W$ is proper in $G$. This concludes the proof.  
\end{proof}

We use the following well-known lemma, which we prove for completeness. 
\begin{lemma} \label{lem:treecnn} Let $G$ be a connected graph, $a, b, c \in V(G)$ with $d(a) = d(b) = d(c) = 1$, and let $H$ be a connected induced subgraph of $G$ containing $a, b, c$ with $V(H)$ minimal subject to inclusion. Then either $H$ is a subdivision of $K_{1,3}$ with $a, b, c$ as the vertices of degree one, or $H$ contains a triangle. 
\end{lemma}
\begin{proof}
Let $G, a, b, c, H$ be as in the statement of the theorem. Let $P$ be a shortest $a-b$-path in $H$, and let $Q$ be a shortest path from $c$ to a vertex $d$ with a neighbor in $V(P)$. By the minimality of $V(H)$, it follows that $V(H) = V(P) \cup V(Q)$. Moreover, $P$ and $Q$ are induced paths and no vertex of $Q \setminus d$ has a neighbor in $V(P)$. If $d$ has exactly one neighbor in $V(P)$, then the result follows. If $d$ has two consecutive neighbors in $V(P)$, then $H$ contains a triangle. Otherwise, let $w \in V(P)$ such that $d$ has a neighbor both on the subpath of $P$ from $w$ to $a$ and on the subpath of $P$ from $w$ to $b$. It follows that $w \not\in \sset{a,b,c}$, and that $H \setminus w$ is connected and contains $a, b, c$. This contradicts the minimality of $V(H)$, and the result follows. 
\end{proof}

\begin{lemma} \label{lem:contract} Let $G$ be an $\isktk$-free graph, $s$ the center of a proper wheel in $G$, $K$ a component of $G \setminus N[s]$, and $N$ the set of vertices in $N(s)$ with a neighbor in $K$. Let $G'$ arise from $G$ by contracting $V(K)$ to a new vertex $z$. If $G|(V(K) \cup N \cup \sset{s})$ is series-parallel, then $G'$ is $\isktk$-free.  
\end{lemma}
\begin{proof}
$G'$ does not contain a triangle, because $N_{G'}(z) \subseteq N_G(s)$ is stable, and hence $z$ is not in a triangle in $G'$. Suppose that $H$ is an induced subgraph of $G'$ which is a $K_{3,3}$ or a subdivision of $K_4$. Then $z \in V(H)$. If $z$ has degree two in $H$ (and so $H$ is an $\isk$), let $a, b$ denote its neighbors; we can replace $a-z-b$ by an $a-b$-path $P$ with interior in $K$ and obtain a subdivision of $H$, which is an $\isk$, as an induced subgraph of $G$, a contradiction. Thus $z$ has degree three in $H$; let $a, b, c$ denote the neighbors of $z$ in $H$. Let $P$ be a shortest $a-b$-path with interior in $K$. Then $c$ has at most one neighbor on $P$, for otherwise $G|(V(P) \cup \sset{a,b,c,s})$ is a wheel, contrary to the fact that $G|(V(K) \cup N \cup \sset{s})$ is series-parallel and does not contain a wheel by Theorem~\ref{thm:duffin}. Let $Q$ be a shortest path from $c$ to $V(P) \setminus \sset{a,b}$ with interior in $K$; then each of $a, b, c$ has a unique neighbor in $V(Q) \cup V(P)$ by symmetry. Let $H'$ be a minimal connected induced subgraph of $G|(V(P) \cup V(Q))$ containing $a, b, c$. Since $G|(V(K) \cup N \cup \sset{s})$ is series-parallel, it follows that $H'$ is a subdivision of $K_{1,3}$ with $a, b, c$ as the vertices of degree one by Lemma~\ref{lem:treecnn}. Therefore, $G|(V(H\setminus z) \cup V(H'))$ is a subdivision of $H$, and by Theorem~\ref{lem:sub}, it contains an $\isk$ or a $K_{3,3}$ subgraph in $G$. This is a contradiction, and the result is proved. 
\end{proof}

\begin{lemma} \label{lem:non-centers} Let $G$ be an $\isktk$-free graph, $s$ the center of a proper wheel in $G$, $K$ a component of $G \setminus N[s]$, and $N$ the set of vertices in $N(s)$ with a neighbor in $K$, and let $H = G|(V(K) \cup N \cup \sset{s})$ be series-parallel. Let $G'$ arise from $G$ by contracting $V(K)$ to a new vertex $z$. Then $z$ is not the center of a proper wheel in $G'$, and for $v \in V(G') \setminus \sset{s,z}$, if $v$ is the center of a proper wheel in $G'$, then $v$ is the center of a proper wheel in $G$.
\end{lemma}
\begin{proof} 
Since $N_{G'}(z) \subseteq N_{G'}(s)$, it follows that $z$ is not the center of a proper wheel in $G'$, for otherwise $s$ would have a neighbor in every sector of such a wheel. This proves the first statement of the lemma. 

Throughout the proof, let $v \in V(G') \setminus \sset{s,z}$ be the center of a proper wheel in $G'$, and let $W = (C, v)$ be such a wheel with a minimum number of spokes. Since $G'$ is $\isktk$-free by Lemma~\ref{lem:contract}, it follows that $W$ satisfies the hypotheses of  Theorem~\ref{wheelmain0}. Our goal is to show that $v$ is the center of a proper wheel in $G$. 

\vspace*{-0.4cm}\begin{equation} \label{eq:zinC}
  \longbox{\emph{If $z \in V(C)$, then $v$ is the center of a proper wheel in $G$.}}
\end{equation} 

Suppose that $z \in V(C)$. Let $a, b$ denote the neighbors of $z$ in $V(C)$. Let $P$ be a shortest $a-b$-path with interior in $K$. Then every vertex in $V(K)$ has at most two neighbors in $V(P)$. Let $W' = (C',v)$ be the wheel in $G$ that arises from $W$ by replacing the subpath $a-z-b$ of $C$ by $a-P-b$ to obtain $C'$. 

It remains to show that $W'$ is a proper wheel in $G$. Suppose that some vertex $x \in V(G) \setminus V(K)$ has two or more neighbors in $P^*$. Then $x \in N \subseteq V(H)$, and $H|(\sset{a,b,x,s} \cup V(P))$ is a wheel in $H$ with center $x$, a contradiction since $H$ is series-parallel Theorem~\ref{thm:duffin}. 

Since $v \not\in V(K)$, it follows from the claim of the previous paragraph that $v$ has at most one neighbor in $P^*$, and no neighbor unless $v$ is adjacent to $z$, and therefore there are at most two sectors of $W'$ intersecting $P^*$. We claim that if for a vertex $x$ we have $|N_G(x) \cap V(C')| \geq 3$, then $x\not\in V(K)$ and $|N_{G'}(x) \cap V(C)| \geq 3$. Suppose that $x$ is a vertex violating this claim. If $x \in K$, then $N_G(x) \cap V(C') \subseteq V(P)$, and so $|N_G(x) \cap V(C')| \leq 2$ by the minimality of $|V(P)|$, a contradiction; it follows that $x \not\in K$. Therefore, $|N_G(x) \cap P^*| \leq 1$, and thus $|N_G(x) \cap V(C')|- |N_{G'}(x) \cap V(C)| \leq 1$. But $|N_G(x) \cap V(C')| > 3$ by Lemma~\ref{nolink}, and so $|N_{G'}(x) \cap V(C)| \geq 3$, a contradiction. So the claim holds.  

Now suppose that there is a vertex $x$ which is not proper for $W'$. If $x$ has neighbors in at most one sector of $W'$, then $|N_G(x) \cap V(C')| \geq 3$, but we proved above that $|N_{G'}(x) \cap V(C)| \geq 3$, and so, since $W$ is proper, $x$ is adjacent to $v$, a contradiction. It follows that $x$ has neighbors in more than one sector of $W'$. Since $x$ is proper for $W$, it follows that $x$ has a neighbor in $P^*$ and thus, either $x \in V(K)$ or $x$ is adjacent to $z$. Since $x$ is proper for $W$, it follows that $N_G(x) \cap V(C')$ is contained in the sectors of $W'$ intersecting $P^*$. In particular, there are exactly two such sectors $S_1$ and $S_2$ of $W'$, they are consecutive, and $v$ has a neighbor in $P^*$. Consequently, $v$ is adjacent to $z$ and $z$ is a spoke in $W$. 

We claim that $x$ has at most two neighbors in $V(C')$. If $x \in V(K)$ then $N_G(x) \cap V(C') \subseteq V(P)$ and we have already shown that every vertex of $K$ has at most two neighbors in $P$. Thus we may assume that $x\not\in K$, and so $x$ is adjacent to $z$. Since $G'$ is triangle-free by Lemma~\ref{lem:contract}, it follows that $x$ is not adjacent to $v$. Since $x$ is proper for $W$, it follows that $x$ has at most two neighbors in $V(C)$, and hence in $V(C')$, by our first claim. This proves our second claim. It follows that $x$ has exactly one neighbor $s_1$ in $S_1 \setminus S_2$ and exactly one neighbor $s_2$ in $S_2 \setminus S_1$. If $x$ is non-adjacent to $v$, then $G|(V(S_1) \cup V(S_2) \cup \sset{x, v})$ is an $\isk$ in $G$, a contradiction. Therefore, $x$ is adjacent to $v$ and can be linked to the cycle $G|(V(S_1) \cup \sset{v})$ via $x-v$, $x-s_1$, and a subpath of $x-s_2-S_2$. Therefore $W'$ is a proper wheel in $G$. This proves~\eqref{eq:zinC}. 

\bigskip

By~\eqref{eq:zinC}, we may assume that $z\not\in V(C)$. So $W$ is a wheel in $G$. Since $W$ is proper in $G'$, there is a sector $S$ of $W$ containing all neighbors of $z$ in $C$. Then clearly the following holds. 

\vspace*{-0.4cm}\begin{equation} \label{eq:kristina1}
  \longbox{\emph{For every $x \in K$, $N_G(x) \cap V(C) \subseteq N_{G'}(z) \cap V(C) \subseteq V(S)$.}}
\end{equation} 

Next we claim the following.

\vspace*{-0.4cm}\begin{equation} \label{eq:kristina2}
  \longbox{\emph{If $z$ is not adjacent to $v$, then $W$ is a proper wheel in $G$.}}
\end{equation} 

If $x \in G \setminus (V(C) \cup V(K))$, then $x$ is proper for $W$ is in $G$ as $x$ is proper for $W$ in $G'$. Now consider a vertex $x \in V(K)$. Since $z$ is not adjacent to $v$, and $W$ is proper in $G'$, it follows that $|N_{G'}(z) \cap V(C)| \leq 2$. Then by~\eqref{eq:kristina1}, $|N_G(x) \cap V(C)| \leq 2$, and hence $x$ is proper for $W$ in $G$. This proves~\ref{eq:kristina2}. 

\bigskip

By~\eqref{eq:kristina2}, we may assume that $z$ is adjacent to $v$. Let $a$ and $b$ be the ends of $S$. We now define a sequence of wheels in $G$ with center $v$. Let $W_1 = W$ and $S_1 = S$. Assume that wheels $W_1, \dots, W_i$ have been defined, and define $W_{i+1}$ as follows. If there is a vertex $x_i \in V(K)$ that is not adjacent to $v$ and has at least three neighbors in $S_i$, then let $S_{i+1}$ be the path from $a$ to $b$ in $G|(V(S_i) \cup \sset{x_i})$ that contains $x_i$, and (by~\eqref{eq:kristina1}) let $W_{i+1}$ be the wheel obtained form $W_i$ by replacing $S_i$ by $S_{i+1}$. Since $S_{i+1}$ is strictly shorter than $S_i$, this sequence must stop at some point; say it stops with wheel $W_t$. For $1 \leq i \leq t$, let $C_i$ be the rim of $W_i$ (so $V(C_i) = (V(C) \setminus V(S)) \cup V(S_i)$). Then $W_t = (C_t, v)$ is a wheel in $G$ such that every vertex of $K$ that has at least three neighbors in $S_t$ is adjacent to $v$. We will show that $W_t$ is a proper wheel in $G$, but first we show the following. 



\vspace*{-0.4cm}\begin{equation} \label{eq:induction}
  \longbox{\emph{For $1 \leq i < t$, if a vertex $y$ is proper for $W_i$, then $y$ is proper for $W_{i+1}$.}}
\end{equation} 

Suppose that $y$ is proper for $W_i$ and not proper for $W_{i+1}$. Then $y$ is adjacent to $x_i$. Suppose first that $y$ is non-adjacent to $v$ and $|N_G(y) \cap V(C_{i+1})| \geq 3$. Since $y$ cannot have three neighbors in $C_{i+1}$ by Lemma~\ref{nolink}, it follows that $|N_G(y) \cap V(C_{i+1})| > 3$. Moreover, since $N_G(y) \cap V(C_{i+1}) \subseteq \sset{x_i} \cup (N_G(y) \cap V(C_i))$, it follows that $|N_G(y) \cap V(C_i)| \geq 3$. But then $y$ is not proper for $W_i$, a contradiction. It follows that $y$ has a neighbor in $C_{i+1} \setminus S_{i+1} = C \setminus S$, and thus $y \not\in V(K)$. Therefore $y \in V(G')$, and since $y$ is adjacent to $x_i$ in $G$, it follows that $y$ is adjacent to $z$ in $G'$. Since $z$ is adjacent to $v$ in $G'$ and $G'$ is triangle-free by Lemma~\ref{lem:contract}, it follows that $y$ is non-adjacent to $v$. Note that since $i+1>1$, it follows that a vertex of $K$ has a neighbor in $S^*$, and therefore $z$ has a neighbor in $S^*$. Since $W$ satisfies the hypotheses of Theorem~\ref{wheelmain0}, and since $y-z$ is a path containing exactly one neighbor of $v$, it follows that the neighbors of $y$ in $C$ are in a sector $S'$ of $W$ consecutive with $S$. Since $y$ is non-adjacent to $v$, it follows that $y$ has at most two neighbors in $S'$. Note that since $G'$ is triangle-free, and $N_{G'}(z) \cap V(C) \subseteq V(S)$, it follows that $z$ has no neighbors in $S'$. If $y$ has exactly two neighbors in $S'$, then $y$ can be linked in $G'$ to the hole $v - S' - v$ via two one-edge paths and the path $y-z-v$. So $y$ has exactly one neighbor $r$ in $S'$ that is in $V(S') \setminus V(S)$, and now $z$ can be linked to $v-S'-v$ via the paths $z-v$, $z-y-r$, and a path with interior in $S$, contrary to Lemma~\ref{nolink}. This concludes the proof of~\eqref{eq:induction}. 

\bigskip
Every vertex in $G \setminus V(K)$ is proper for $W$ and hence it is proper for $W_t$ by~\eqref{eq:induction}. Suppose that there is a vertex $x \in V(K)$ that is not proper for $W_t$. By~\eqref{eq:kristina1}, $N_G(x) \cap V(C_t) \subseteq V(S_t)$. So $x$ is non-adjacent to $v$ and has at least three neighbors in $S_t$, contradicting the assumption that the wheel sequence terminates with $W_t$. Therefore, $W_t$ is a proper wheel in $G$ with center $v$. 

\end{proof}

\section{Tools}

In this section we develop tools for our main theorem for finding a vertex of degree one, or a cycle with all but a few vertices of degree two. 


\begin{lemma} \label{lem:tree} Let $G$ be a graph, $x \in V(G)$, such that $G \setminus x$ is a forest. Then either $V(G) = N[x]$ and $G \setminus x$ is stable, or $V(G) \setminus N[x]$ contains a vertex of degree at most one in $G$, or $G$ contains an induced cycle $C$ containing $x$ such that every vertex of $V(C) \setminus \sset{x}$ except for possibly one has degree two in $G$.
\end{lemma}
\begin{proof}
If every component of $G\setminus x$ contains exactly one vertex, then either $V(G) = N[x]$ or $V(G) \setminus N[x]$ contains a vertex of degree zero. Hence, we may assume that there exists a component $T$ of $G\setminus x$ with at least two vertices, and $T$ is a tree. Let $A$ be the set of vertices of degree at least three in $T$. If $A$ is non-empty, then let $T'$ be the subtree of $T$ that contains all vertices of $A$ and minimal with respect to this property, and let $a$ be a leaf of $T'$. There is a path $P = v - \dots - v'$ in $T$, whose ends are distinct leaves of $T$ and $P$ contains at most one vertex of degree three in $T$ (namely $a$). This is trivial is $A$ if empty, and follows from the definition of $a$ otherwise.

If $x$ is non-adjacent to $v$, then $v$ is a vertex in $V(G) \setminus N[x]$ of degree one in $G$, so we may assume that $x$ is adjacent to $v$, and similarly for $v'$. Now, let $v''$ be the neighbor of $x$ in $P\setminus v$ closest to $v$ along $P$. We set $C = x - v-P-v''-x$ and observe that all vertices of $C$ except possibly $x$ and $a$ have degree two in $G$.
\end{proof}

\begin{lemma} \label{lem:farcycle} Let $G$ be a series-parallel graph, and let $x, y \in V(G)$ with $x=y$ or $xy \in E(G)$. If $G \setminus \sset{x,y}$ contains a cycle, then there is an induced cycle $C$ in $G$ such that $V(C) \cap \sset{x,y} = \emptyset$ and all but at most two vertices of $C$ have degree two in $G$ (and are thus anticomplete to $\sset{x,y}$), or $V(G) \setminus (N[x] \cup N[y])$ contains a vertex of degree at most one in $G$. 
\end{lemma}
\begin{proof}
By contracting the edge $xy$ and deleting any parallel edges that may arise, we may assume that $x=y$. We may further assume that every vertex except for possibly $x$ has degree at least two, because vertices of degree one in $N(x)$ can be deleted without affecting the hypotheses or the conclusions, and if there is a vertex of degree at most one in $V(G) \setminus N[x]$, then the conclusion holds. 

Let $C$ be a cycle in $G \setminus x$. Since $G$ is series-parallel and by the definition of series-parallel graphs, it follows that there do not exist three paths from $x$ to $V(C)$ that are vertex disjoint except for $x$ in $G$. By Menger's theorem \cite{menger}, it follows that there is a partition  $(X, Y, Z)$ of $V(G)$ with $X$ of size at most two, and $Y,Z \neq \emptyset$ such that $Y$ is anticomplete to $Z$ in $G$, $V(C) \subseteq Y \cup X$ and $x \in Z$.

We choose a partition $(X, Y, Z)$ with $|X|$ minimal, and subject to that, $|X \cup Y|$ minimal, such that $Y$ is anticomplete to $Z$ in $G$, $Y, Z \neq \emptyset$, $x \in Z$, and $G|(Y \cup X)$ contains a cycle. It follows that $|X| \leq 2$. 

Suppose first that $X = \emptyset$. If $G|Y$ is an induced cycle, the result follows. Otherwise, since $G|Y$ contains a cycle, it follows that there is a vertex $x'$ such that $G|(Y \setminus \sset{x'})$ contains a cycle. By induction applied to $G|Y$ and the vertex $x'$, the result follows. 

Next, suppose that $X = \sset{x'}$. If $G|Y$ is a forest, then $x' \neq x$ and thus we obtain the desired result by applying Lemma~\ref{lem:tree} to $G|(X \cup Y)$. Otherwise, we apply induction to $G|(X \cup Y)$ and $x'$, and again, the result follows. 

It follows that $X = \sset{x', y'}$, and therefore, the component of $G|(Z \cup X)$ containing $x$ contains $x'$ and $y'$, for otherwise $\sset{x'}$ or $\sset{y'}$ would be a better choice of $X$ for the partition. Suppose that $G|Y$ is connected. If there is a vertex $z$ such that every $x'-y'$-path with interior in $G|Y$ uses $z$, then $\sset{x',z}$ or $\sset{y', z}$ yields a better choice of $X$ and partition. Therefore, by Menger's theorem \cite{menger}, there are two disjoint paths $P_1, P_2$ from $x'$ to $y'$ with interior in $Y$, and since $G|Y$ is connected, it follows that there is path $Q$ from $P_1$ to $P_2$ in $Y$. Moreover, there is a path $R$ from $x'$ to $y'$ with interior in $Z$ since the component of $G|(Z \cup X)$ containing $x$ also contains $x'$ and $y'$; but $P_1 \cup P_2 \cup Q \cup R$ is a (not necessarily induced) subdivision of $K_4$ in $G$, contrary to the fact that $G$ is series-parallel. Thus $G|Y$ is not connected. By the minimality of $X \cup Y$, for every component $K$ of $G|Y$, the graph $G|((X \setminus \sset{x}) \cup V(K))$ contains no cycle. However, $G|(X \cup Y)$ contains a cycle $C$ not using $x$, and so $C$ contains vertices from more than one component of $G|Y$. It follows that $x', y' \in V(C)$, and thus $x \not\in \sset{x', y'}$. Therefore, for every component $K$ of $G|Y$, the graph $G|(X \cup V(K))$ is a tree. Since $K$ is connected, it follows that $x', y'$ are leaves. If $G|(X \cup V(K))$ contains a leaf other than $x', y'$, then the result follows. So each component is a path from $x'$ to $y'$, and no vertex of the path except for $x', y'$ has further neighbors in $G$. But then the union of two of those paths (there are at least two, since $G|Y$ is not connected) yields the desired cycle; the result follows. 
\end{proof}

\begin{theorem} \label{thm:girth} Let $G$ be an $\isktk$-free graph, $x, y \in V(G)$ with $x=y$ or $xy \in E(G)$. Then either 
\begin{itemize}
\item $V(G) = N[x] \cup N[y]$;
\item there exists a vertex in $V(G) \setminus (N[x] \cup N[y])$ of degree at most one in $G$;
\item there exists an induced cycle $C$ containing at least one of $x, y$ such that at most one vertex $v$ in $V(C) \setminus (N[x] \cup N[y])$ has $d(v) > 2$; or 
\item there exists an induced cycle $C$ containing neither $x$ nor $y$ and a vertex $z \in V(C)$ such that at most one vertex $v$ in $V(C) \setminus N[z]$ has $d(v) > 2$. 
\end{itemize}
\end{theorem}
\begin{proof}
Suppose first that $G$ is series-parallel. Define $H=G$ and $v =x$ if $x=y$, and define $H$ as the graph that arises from contracting the edge $xy$ to a new vertex $v$ if $x \neq y$. Then $H$ is series-parallel. Suppose that $H \setminus v$ is a forest, and apply Lemma~\ref{lem:tree}. If the first outcome of Lemma~\ref{lem:tree} holds, then $V(H) = N_H(v)$, and so $V(G) = N_G(x) \cup N_G(y)$. If the second outcome of Lemma~\ref{lem:tree} holds, then $V(H) \setminus N_H[v]$ contains a vertex of degree at most one in $H$, and so $V(G) \setminus (N_G[x] \cup N_G[y])$ contains a vertex of degree at most one in $G$. Finally, if the third outcome of Lemma~\ref{lem:tree} holds, then $H$ contains an induced cycle $C$ containing $v$ such that every vertex of $V(C) \setminus \sset{v}$ except for at most one has degree two in $H$, and so there is an induced cycle $C'$ in $G$ containing at least one of $x, y$ such that every vertex of $V(C') \setminus \sset{x,y}$ except for possibly one has degree two in $G$. This proves the result in the case that $H \setminus v$ is a forest. So $H \setminus v$ contains a cycle, and thus $G \setminus \sset{x,y}$ contains a cycle. By Lemma~\ref{lem:farcycle}, either $V(G) \setminus (N[x] \cup N[y])$ contains a vertex of degree at most one in $G$, or $G$ contains a cycle $C$ with $V(C) \cap \sset{x,y} = \emptyset$ and such that all but at most two vertices in $V(C)$ have degree two in $G$. In the former case, the second outcome of this theorem holds; in the latter case, the fourth outcome of this theorem holds by choosing $z \in V(C)$ with $d_G(z)$ maximum among vertices in $V(C)$. 

Thus we may assume that $G$ contains a proper wheel by Lemma~\ref{proper}; let $z$ be the center of a proper wheel (where possibly $z \in \sset{x,y}$). Let $W$ be such a wheel with minimum number of spokes.  Let $Z = \sset{x,y} \cap N(z)$. Since $x=y$ or $xy \in E(G)$, it follows that $x$ and $y$ are in the same component of $G \setminus (N[z] \setminus Z)$. Since $N(z)$ is stable, it follows that $|Z| \leq 1$. Therefore, by Theorem~\ref{wheelmain0}, the component of $G \setminus (N[z] \setminus Z)$ containing $\sset{x,y}\setminus \sset{z}$ includes the interiors of at most two sectors of $W$. Again by Theorem~\ref{wheelmain0}, the interior of every other sector of $W$ is contained in a separate component of $G \setminus (N[z] \setminus Z)$. Since $W$ has at least four sectors by Lemma~\ref{nolink}, there is a component $K$ of $G \setminus (N[z] \setminus Z)$ that does not contain $x$ and $y$, and that contains no neighbor of $x, y$. Let $N$ be the set of neighbors of $z$ with a neighbor in $K$. Then, we apply induction to $H=G|(V(K) \cup N \cup \sset{z})$ and $z$. By the choice of $H$ and $z$, the first outcome does not hold. If the second outcome holds for $H$ and $z$, then it holds for $G$ and $x,y$ as well, since $(N[x] \cup N[y]) \cap V(H) \subseteq N[z] \cap V(H)$. If the third or fourth outcome holds for $H$ and $z$, then the third or fourth outcome holds for $G$ and $x,y$. 
\end{proof}

\section{Main Result}

We say that $(G, x, y)$ has the property $\mathcal{P}$ if $V(G) \setminus (N[x] \cup N[y])$ contains a vertex of degree at most two in $G$.




We can now prove Theorem~\ref{thm:main0} which we restate:

\begin{theorem} \label{thm:main} Let $G$ be an $\isktk$-free graph which is not series-parallel, and let $(x,y)$ be a non-center pair for $G$. Then $(G,x,y)$ has the property $\mathcal{P}$. 
\end{theorem}

\begin{proof}
Suppose for a contradiction that the theorem does not hold, and let $(G,x,y)$ be a counterexample with $|V(G)|$ minimum. Then every vertex in $V(G) \setminus (N[x] \cup N[y])$ has degree at least three in $G$. Since $G$ is not series-parallel, and $G$ is $\isktk$-free, it follows from Theorem~\ref{lem:sub} that $G$ contains a wheel and hence by Lemma~\ref{proper}, it follows that $G$ contains a proper wheel $W = (C,s)$. Let $C_1, \dots, C_k$ denote the components of $V(G) \setminus N[s]$. For $i = 1, \dots, k$ let $N_i$ denote the set of neighbors $v$ of $s$ such that $v$ has a neighbor in $C_i$, and let  $G_i$ denote the induced subgraph of $G$ with vertex set $V(C_i) \cup N_i \cup \sset{s}$. 

\vspace*{-0.4cm}\begin{equation} \label{eq:spcycle}
  \longbox{\emph{For $i = 1, \dots, k$, if $G_i$ is series-parallel and $G_i \setminus (N[s] \cap \sset{x,y,s})$ contains a cycle, then $\sset{x,y} \cap V(C_i) \neq \emptyset$.}}
\end{equation}

Let $i \in \sset{1,\dots, k}$ such that $G_i$ is series-parallel, and let $G_i \setminus (N[s] \cap \sset{x,y,s})$ contain a cycle. Since $G$ is triangle-free, it follows that $1 \leq |N[s] \cap \sset{x,y,s}| \leq 2$. By Lemma~\ref{lem:farcycle} applied to $G_i$ and the vertices in $N[s] \cap \sset{x,y,s}$, it follows that either there is a vertex in $V(G_i) \setminus N[s]$ of degree at most one in $G_i$ anticomplete to $\sset{x,y} \cap N(s)$, or $G_i \setminus (N[s] \cap \sset{x,y,s})$ contains a cycle $C'$ with at least two vertices of degree two in $G_i$. In both cases, there is a vertex $z$ in $V(G_i) \setminus N[s]$ of degree at most two in $G_i$ and $z$ is anticomplete to $N[s] \cap \sset{x,y,s}$, and hence its degree in $G$ is also at most two. Since $(G, x, y)$ does not satisfy property $\mathcal{P}$, it follows that $z \in N[x] \cup N[y]$, and thus $\sset{x,y} \cap V(C_i) \neq \emptyset$. This proves~\eqref{eq:spcycle}. 

\vspace*{-0.4cm}\begin{equation} \label{eq:sp}
  \longbox{\emph{For $i = 1, \dots, k$, if $G_i$ is series-parallel, then $|V(C_i)| = 1$.}}
\end{equation}

Let $i \in \sset{1,\dots, k}$ be such that $G_i$ is series-parallel, and suppose that $|V(C_i)| > 1$. Let $G'$ be the graph that arises from $G$ by contracting $V(C_i)$ to a new vertex $z$. We let $x' = z$ if $x \in V(C_i)$ and $x' = x$ otherwise; and we let $y' = z$ if $y \in V(C_i)$ and $y' = y$ otherwise.  By Lemma~\ref{lem:contract}, $G'$ is $\isktk$-free. By Lemma~\ref{lem:non-centers}, $(x',y')$ is a non-center pair for $G'$. By the minimality of $|V(G)|$, it follows that $(G', x', y')$ has the property $\mathcal{P}$. Let $v \in V(G') \setminus (N[x'] \cup N[y'])$ be a vertex of degree at most two in $G'$. From the definition of $x'$ and $y'$, it follows that $v \not\in N[x] \cup N[y]$. It follows that either $v=z$, or $v\neq z$ and $d_G(v) > 2$, and so $v \in N[z]$.

Suppose first that $v = z$. Then $z \not\in N[x'] \cup N[y']$, and so $V(G_i) \cap \sset{x, y} = \emptyset$. By~\eqref{eq:spcycle}, it follows that $G_i \setminus s$ is a tree. Since $v$ has degree at most two in $G'$, it follows that $|N_i| \leq 2$, and since $G_i \setminus N[s]$ is connected, it follows that every vertex of $N_i$ is a leaf of $G_i \setminus s$.  Thus, either $V(C_i)$ contains a leaf of $G_i \setminus s$, or $G_i \setminus s$ is a path with ends in $N_i$, and so in both cases $V(C_i)$ contains a vertex of degree at most two in $G$. This is a contradiction since $V(G_i) \cap \sset{x,y} = \emptyset$; it follows that $v \neq z$. 

It follows that $v \in N(z)$. Since $d_G(v) > 2$, it follows that $d_{G'}(v) < d_G(v)$, and thus $v$ has more than one neighbor in $V(C_i)$. Let $P$ be a path in $C_i$ between two neighbors of $v$, then $v-P-v$ is a cycle in $G_i \setminus (N[s] \cap \sset{x,y,s})$. By~\eqref{eq:spcycle}, it follows that $V(C_i) \cap \sset{x, y} \neq \emptyset$. But then $z \in \sset{x', y'}$, and so $v \in N[x'] \cup N[y']$, a contradiction.  This proves~\eqref{eq:sp}. 

\vspace*{-0.4cm}\begin{equation} \label{eq:big}
  \longbox{\emph{For $i = 1, \dots, k$, if $G_i$ contains a wheel, then $x \in V(C_i)$ or $y \in V(C_i)$.}}
\end{equation} 

Suppose not, and let $i \in \sset{1, \dots, k}$ be such that $G_i$ contains a wheel and $V(C_i) \cap \sset{x,y} = \emptyset$. Since $N_i$ is a stable set, it follows that $|N_i \cap \sset{x,y}| \leq 1$, and by symmetry, we may assume that $y \not\in N_i$. Let $y' = s$, and let $x' = x$ if $x \in N_i$ and $x' = s$ otherwise. By Lemma~\ref{lem:starcutset}, $(x',y')$ is a non-center pair for $G_i$. Since $G_i$ is an induced subgraph of $G$, it follows that $G_i$ is $\isktk$-free. Since $G$ is a minimum counterexample, it follows that $(G_i, x', y')$ has the property $\mathcal{P}$. Let $v$ be a vertex of degree at most two in $G_i$ with $v \not\in N[x'] \cup N[y']$. Since $v \not\in N[s]$, it follows that $d_G(v) = d_{G_i}(v)$. Therefore $v \in N[x] \cup N[y]$. Let $w \in \sset{x, y}$ be such that $v \in N[w]$. By assumption, $w \not\in V(C_i)$. If $w \in V(G_i)$, then $w \in N_i$, and so $w=x=x'$, and thus $v \not\in N[w]$, a contradiction. It follows that $w \not\in V(G_i)$, and so $N[w] \cap V(G_i) \subseteq N[s]$, and again $v \not\in N[w]$, a contradiction. This proves that $\sset{x,y} \cap V(C_i) \neq \emptyset$, and~\eqref{eq:big} follows. 

\bigskip
It follows from Theorem~\ref{thm:duffin} together with~\eqref{eq:big} and~\eqref{eq:sp} that there is at most one $i \in \sset{1, \dots, k}$ with $|V(C_i)| > 1$. We may assume $|V(C_i)| = 1$ for all $i \in \sset{1, \dots, k-1}$. 

\vspace*{-0.4cm}\begin{equation} \label{eq:girth8}
  \longbox{\emph{If $|V(C_k)| > 1$, let $G' = G \setminus (V(C_k) \cup \sset{s})$; otherwise let $G' = G \setminus s$. Then $G'$ has girth at least eight.}}
\end{equation} 
Observe that $G'$ is bipartite with one side of the bipartition being $N(s)$. 

 Suppose that $C$ is a cycle of length four in $G'$. Let $V(C) = \sset{a,b,c,d}$ and $N(s) \cap V(C) = \sset{a,c}$. If $d_G(b) \neq 2$, let $e$ be a neighbor of $b$ which is not $a,c$. Note that $e \in N(s)$. Then $\sset{a,b,c,d,e,s}$ induces an $\isk$ or a $K_{3,3}$, a contradiction. It follows that $d_G(b) = d_G(d) = 2$, and moreover, $\sset{x,y} \cap \sset{a,c} \neq \emptyset$. So, by symmetry, say $x = a$, and we may assume that $d \neq y$.

Observe that $G \setminus b$ is not series-parallel by Theorem~\ref{thm:duffin}. By the minimality of $|V(G)|$, it follows that there exists $v \in V(G) \setminus (N[x] \cup N[y])$ with $d_{G \setminus \sset{b}}(v) \leq 2$. Since $d_G(v') = d_{G \setminus \sset{b}}(v')$ for all $v' \in V(G) \setminus \sset{x, b, c}$, it follows that $v=c$, and so $N_G(c) = \sset{b,d,s}$, and so $\sset{s, a}$ is a cutset in $G$. Let $G'' = G \setminus \sset{b,c,d}$, and if $y \in \sset{b,c,d}$, let $y' = x$, otherwise, $y' = y$. Then $G''$ is not series-parallel. A proper wheel in $G''$ is proper in $G$, because each vertex in $\sset{b,c,d}$ has at most one neighbor in the wheel, $s$ or $a$. Therefore, $(x,y')$ is a non-center pair for $G''$.  By the minimality of $|V(G)|$, it follows that $(G'', x, y')$ has the property $\mathcal{P}$. But this is a contradiction, since every vertex in $V(G'') \setminus N[x]$ has the same degree in $G$ and $G''$. This proves that $G'$ contains no 4-cycle. 

Suppose $G'$ contains a 6-cycle $C$. Then, since exactly three vertices in $V(C)$ are neighbors of $s$, it follows that $G|(V(C) \cup \sset{s})$ is an $\isk$, a contradiction. It follows that $G'$ has girth at least eight, and so~\eqref{eq:girth8} is proved. 

\vspace*{-0.4cm}\begin{equation} \label{eq:singletons}
  \longbox{\emph{$|V(C_k)| > 1$.}}
\end{equation} 

Suppose not, and let $G' = G \setminus s$. Then $G'$ satisfies the hypotheses of Theorem~\ref{thm:girth}. Since $s$ is the center of a proper wheel in $G$, it follows that there exists a vertex $z$ in $G'$ that is not in $N[x] \cup N[y] \cup N[s]$, and so the first outcome of Theorem~\ref{thm:girth} does not hold. The second outcome does not hold, because every vertex in $V(G') \setminus (N[x] \cup N[y])$ of degree one in $G'$ has degree at most two in $G$, a contradiction. 

Therefore, the third or fourth outcome of Theorem~\ref{thm:girth} holds, and hence there exists an induced cycle $C$ in $G'$ with vertices $c_1 - \ldots - c_t - c_1$, and $i,j \in \sset{1, \dots, t}$, $l \in \sset{0, \dots, 3}$ such that all vertices of $C$ except for $c_i, \dots, c_{i+l}$ (where $c_{t+1} = c_1$ and so on) and $c_j$ have degree two in $G'$, do not coincide with $x, y$ and are non-neighbors of $x, y$. By~\eqref{eq:girth8}, $t \geq 8$. Consequently, $G'$ contains two adjacent vertices  in $V(G') \setminus (N[x] \cup N[y])$ of degree two in $G'$. Since $G$ is triangle-free, it follows that one of them has degree two in $G$, a contradiction. Thus, $|V(C_k)| > 1$, and~\eqref{eq:singletons} is proved. 

\bigskip

By~\eqref{eq:sp}, \eqref{eq:big} and~\eqref{eq:singletons} we may assume that $x \in V(C_k)$. Let $G'$ arise from $G$ by contracting $V(C_k) \cup N_k$ to a single vertex $z$, and by deleting $s$ and every vertex that is only adjacent to $z$. It follows that $G'$ is bipartite. Our goal is to prove that $G' \setminus z$ has girth at least 16, see~\eqref{eq:tree}. By~\eqref{eq:girth8}, we know that $G' \setminus z$ has girth at least eight. 

\vspace*{-0.4cm}\begin{equation} \label{eq:uniquenbrs}
  \longbox{\emph{Every vertex in $V(G') \setminus \sset{z}$ has at most one neighbor in $N_k$ in $G$. There is no 4-cycle in $G'$ containing $z$.}}
\end{equation} 
Suppose first that there is a vertex $v \in V(G') \setminus \sset{z}$ with at least two neighbors $a, b \in N_k$ in $G$. Since $v \in V(G')$ and in $G'$ there are no vertices of degree one adjacent to $z$, it follows that $v$ has another neighbor $c \in N(s) \setminus N_k$. Let $P$ be a path connecting $a$ and $b$ with interior in $V(C_k)$. Such a path exists, since $a, b \in N_k$. It follows that $G|(V(P) \cup \sset{a,b,c,v,s})$ is an $\isk$ in $G$, a contradiction. This implies the first statement of~\eqref{eq:uniquenbrs}. 

Suppose that $z$ is contained in a 4-cycle with vertex set $\sset{a,b,c,z}$ in $G'$ such that $a, c \in N_{G'}(z)$. Note that $a, c \not\in N(s)$ and $b \in N(s) \setminus N_k$. By~\eqref{eq:girth8}, $G \setminus (\sset{s} \cup V(C_k))$ contains no 4-cycle, and thus $a$ and $c$ have no common neighbor in $N_k$. Let $a', c'$ be a neighbor of $a$ and $c$ in $N_k$, respectively; $a'$ and $c'$ exists since $a, c \in N_{G'}(z)$. Let $P$ be a shortest path between $a'$ and $c'$ with interior in $C_k$. Since $b \not\in N_k$, it follows that  $b$ is anticomplete to $V(P)$. Therefore, $G|(\sset{a,b,c, s} \cup V(P))$ is an $\isk$ in $G$, a contradiction. This proves~\eqref{eq:uniquenbrs}. 

\vspace*{-0.4cm}\begin{equation} \label{eq:class}
  \longbox{\emph{$G'$ is $\isktk$-free.}}
\end{equation} 
Since $G'$ is bipartite, it follows that $G'$ is triangle-free. Suppose that $G'$ contains an induced subgraph $H$ which is either a $K_{3,3}$ or an $\isk$. Since $G$ is $\isktk$-free, it follows that $z \in V(H)$. Suppose that $z$ has degree two in $H$. By~\eqref{eq:uniquenbrs}, the neighbors of $z$ in $V(H)$ do not have a common neighbor in $N_k$. Let $P$ be a path in $G$ connecting the neighbors of $z$ in $V(H)$ with interior in $V(C_k) \cup N_k$ containing exactly two vertices in $N_k$. Then $G|((V(H) \setminus \sset{z}) \cup V(P))$ is an induced subdivision of $H$ in $G$. By Theorem~\ref{lem:sub}, it follows that $G$ is not $\isktk$-free, a contradiction. 

 It follows that $z$ has degree three in $H$. Let $a, b, c$ be the neighbors of $z$ in $H$. By~\eqref{eq:uniquenbrs}, each of $a,b,c$ has a unique neighbor in $N_k$. Let $a', b', c'$ be neighbors of $a,b,c$ in $N_k$. Let $H'$ be a minimal induced subgraph of $G|(V(C_k) \cup \sset{a, b, c, a',b',c'})$ which is connected and contains $\sset{a,b,c}$. It follows that each of $a, b, c$ has a unique neighbor (namely $a', b', c'$, respectively), in $H'$. By Lemma~\ref{lem:treecnn}, $H'$ is a subdivision of $K_{1,3}$ in which $a, b, c$ are the vertices of degree one. Consequently, $G|(V(H \setminus z) \cup V(H'))$ is an induced subgraph of $G$ which is a subdivision of $H$. But then $G$ is not $\isktk$-free by Theorem~\ref{lem:sub}. Hence $G'$ is $\isktk$-free. This proves~\eqref{eq:class}. 

\vspace*{-0.4cm}\begin{equation} \label{eq:star}
  \longbox{\emph{$G'$ does not contain a proper wheel with center different from $z$.}}
\end{equation} 
Suppose $v \neq z$ is the center of a proper wheel $G'$. By Theorem~\ref{wheelmain0}, there is a component $C$  of $G' \setminus N[v]$ that is disjoint from $N[z]$. Let $N$ denote the set of vertices in $N(v)$ with a neighbor in $C$.

Then $H = G'|(N \cup V(C) \cup \sset{v})$ satisfies the hypotheses of Theorem~\ref{thm:girth}. Since $V(C) \neq \emptyset$, it follows that the first outcome of Theorem~\ref{thm:girth} does not hold. Moreover, every vertex in $V(H) \setminus N[v]$ of degree one in $H$ has degree at most two in $G$, since such a vertex belongs to $C$ and $C$ is disjoint from $N[z]$, and the only additional neighbor that such a vertex may have in $G$ is $s$. Furthermore, such a vertex is in $V(G') \setminus N[z]$, and hence in $V(G) \setminus (N[x] \cup N[y])$ as $x \in V(C_k)$. It follows that the second outcome of Theorem~\ref{thm:girth} does not hold. 

Therefore, the third or fourth outcome of Theorem~\ref{thm:girth} holds, and hence there exists an induced cycle $C'$ in $H$ with vertices $c_1 - \ldots - c_t - c_1$, and $i,j \in \sset{1, \dots, t}$, $l \in \sset{0, \dots, 3}$ such that all vertices of $C'$ except for $c_i, \dots, c_{i+l}$ (where $c_{t+1} = c_1$ and so on) and $c_j$ have degree two in $H$, do not coincide with $v$ and are non-neighbors of $v$. By~\eqref{eq:girth8}, $t \geq 8$, since $z \not\in V(H)$. Consequently, $G'$ contains two adjacent vertices  in $V(G') \setminus N[z]$ of degree two in $G'$. Since $G$ is triangle-free, it follows that one of them is non-adjacent to $s$ and thus has degree two in $G$, a contradiction. Hence~\eqref{eq:star} is proved. 

\vspace*{-0.4cm}\begin{equation} \label{eq:forest}
  \longbox{\emph{For every component $K$ of $G' \setminus N[z]$, $G'|(V(K) \cup N(z))$ is a forest.}}
\end{equation} 
Suppose not, and let $K$ be a component of $G' \setminus N[z]$ such that $G'|(V(K) \cup N(z))$ is not a forest. Suppose first that $H = G'|(V(K) \cup N[z])$ is not series-parallel. Then $H$ contains a proper wheel by Lemma~\ref{proper}. Let $v$ be the center of a proper wheel in $H$. Since $H \setminus N[z]$ is connected, it follows from Theorem~\ref{wheelmain0} that $v \neq z$. By Lemma~\ref{lem:starcutset}, it follows that $v$ is the center of a proper wheel in $G'$, contrary to~\eqref{eq:star}. 

It follows that $H$ is series-parallel, and by our assumption, $H \setminus z$ contains a cycle. By applying Lemma~\ref{lem:farcycle} to $H$ and $z$, it follows that there is either a vertex in $V(H) \setminus N[z]$ of degree one, or a cycle $C$ not containing $z$, with all but at most two vertices of degree two in $H$. In the latter case, since $G' \setminus z$ has girth at least eight, $C$ contains two adjacent vertices in $V(H) \setminus N[z]$ of degree two in $H$, and thus of degree two in $G'$. Since $G$ is triangle-free, it follows that in both cases $G$ contains a vertex of degree at most two not in $N[z]$, and thus not in $N[x] \cup N[y]$. This is a contradiction, and~\eqref{eq:forest} is proved. 

\vspace*{-0.4cm}\begin{equation} \label{eq:tree}
  \longbox{\emph{The girth of $G' \setminus z$ is at least 16.}}
\end{equation} 

Suppose that this is false. Let $C$ be an induced cycle in $G' \setminus z$ of length less than 16. Since by~\eqref{eq:forest}, for every component $K$ of $G' \setminus N[z]$, we have that $G'|(V(K) \cup N(z))$ is a forest, it follows that $C \setminus N[z]$ has at least two components. Since $z$ is not contained in a 4-cycle in $G'$ by~\eqref{eq:uniquenbrs}, and $G'$ is bipartite, it follows that each component of $C \setminus N[z]$ has at least three vertices. If $C \setminus N[z]$ has at least four components, it follows that $C$ has length at least 16. If $C \setminus N[z]$ has exactly three components, then $G|(V(C) \cup \sset{z})$ is an $\isk$, a contradiction. So $C \setminus N[z]$ has exactly two components. For every  component $K$ of $G' \setminus N[z]$, by~\eqref{eq:forest} we have that $G'|(V(K) \cup N(z))$ is a forest. Therefore, the two components of $C \setminus N[z]$ are contained in two different components of $G' \setminus N[z]$; say $A$ and $B$. Let $N_A, N_B$ denote the set vertices in $N(z)$ with a neighbor in $A$, $B$, respectively. Suppose that $|N_A| \geq 3$. Since $|V(C) \cap N(z)| = 2$, it follows that there is a path $P$ from a vertex $c$ in $N_A \setminus V(C)$ to $V(C)$ with interior in $V(A)$. Since $G'|(V(A) \cup N_A)$ and $G'|(V(B) \cup N_B)$ are trees, it follows that $c$ has at most one neighbor in each component $K$ of $C \setminus N[z]$. Therefore, $G'|(V(P) \cup V(C) \cup \sset{z})$ contains an induced subgraph of $G'$ which is either a subdivision of $K_4$ or of $K_{3,3}$, a contradiction by Theorem~\ref{lem:sub} and~\eqref{eq:class}. So $|N_A| = 2$. Since $G'|(V(A) \cup N_A)$ is a tree, it follows that either $A$ contains a vertex of degree one in $G'$, non-adjacent to $z$, or $G'|(V(A) \cup N_A)$ is a path containing at least five vertices, and hence $A$ contains two adjacent vertices of degree two in $G'$, non-adjacent to $z$. Since $G$ is triangle-free, it follows that in either case $G$ contains a vertex of degree at most two not in $N[z]$, and thus not in $N[x] \cup N[y]$. This is a contradiction, and~\eqref{eq:tree} is proved. 

\bigskip

Recall that $\sset{x, y} \cap V(C_k) \neq \emptyset$, and we may assume that $x \in V(C_k)$, and thus $y \in V(C_k) \cup N_k$. Let $G''$ be the graph that arises from $G$ by deleting $\sset{s} \cup (V(C_k) \setminus \sset{x}) \cup (N_k \setminus \sset{y})$, and every vertex other than $x$ with neighbors only in $N_k$ (this last operation does not change the degree of any vertex in $V(G'')$ except for possibly $y$). Then $N_{G''}(x) \subseteq \sset{y}$. It follows from~\eqref{eq:tree} that $G'' \setminus \sset{y}$ has girth at least 16, and from~\eqref{eq:girth8} that $G''$ has girth at least eight. If $y \in V(G'')$, let $y' = y$; otherwise, let $y' = x$. It follows that if $y' = y$, then $y \in N_k$. 

Since $G''$ is an induced subgraph of $G$, it follows that $G''$ and $x, y'$ satisfy the hypotheses of Theorem~\ref{thm:girth}. 

Since $s$ is the center of a proper wheel, it follows from Theorem~\ref{wheelmain0} that there are at least two components of $G'' \setminus N[s]$ in which $y'$ has no neighbors. Consequently, $V(G'') \neq N[x] \cup N[y']$, and thus the first outcome of Theorem~\ref{thm:girth} does not hold. 

The second outcome of Theorem~\ref{thm:girth} does not hold, because if $G''$ contains a vertex $v$ of degree one non-adjacent to $y'$, then $v$ has degree at most two in $G$, and $v \not\in N[x] \cup N[y]$, a contradiction. 

Suppose that the third outcome holds, and so $G''$ contains an induced cycle $C$ containing $y'$ (since $d_{G''}(x) \leq 1$) such that at most one vertex in $V(C) \setminus N[y]$ has degree more than two. Since $G''$ has girth at least eight, $|V(C)| \geq 8$, and in particular $C$ contains a vertex $v$ of distance three from $y$ in $C$ and degree two in $G''$. Let $y-a-b-v$ be the three-edge path from $y$ to $v$ in $C$. Then $v$ is not adjacent to $s$ in $G$, because $G|(V(G'') \cup \sset{s})$ is bipartite and $ys \in E(G)$. Moreover, $v$ anticomplete to $N_k$, because otherwise $z-a-b-v-z$ is a 4-cycle in $G'$ using $z$, contradicting~\eqref{eq:uniquenbrs}. So $v$ has degree two in $G$ and is not in $N[x] \cup N[y]$, a contradiction. 

Thus, the fourth outcome holds, and so $G''$ contains an induced cycle $C$ not containing $x,y'$ and containing a vertex $z'$ such that at most one vertex in $V(C) \setminus N[z']$ has degree more than two in $G''$. Since $|V(C)| \geq 16$, it follows that $C$ contains a path $P = p_1 - \ldots - p_6$ of six vertices, all of degree two in $G''$ and non-adjacent to $x,y'$. We may assume that $N(s) \cap V(P) \subseteq \sset{p_1, p_3, p_5}$ by symmetry. Since $z$ is not in a 4-cycle in $G'$ by~\eqref{eq:uniquenbrs}, not both $p_2$ and $p_4$ have a neighbor in $N_k$. It follows that either $p_2$ or $p_4$ has degree two in $G$, a contradiction. This completes the proof of Theorem~\ref{thm:main}. 
\end{proof}

We can now prove Theorem~\ref{v2} which we restate:

\begin{theorem} \label{thm:main2} Let $G$ be an $\iskt$-free graph. Then either $G$ has a clique cutset, $G$ is complete bipartite, or $G$ has a vertex of degree at most two. 
\end{theorem}
\begin{proof}
Let $G$ be an $\iskt$-free graph. If $G$ is series-parallel, then $G$ contains a vertex of degree at most two Theorem~\ref{thm:duffin}. If $G$ contains $K_{3,3}$ as a subgraph, then by Theorem~\ref{thm:todo}, either $G$ is complete bipartite or $G$ has a clique cutset. If $G$ is not series-parallel and $K_{3,3}$-free, then $G$ contains a vertex of degree at most two by Theorem~\ref{thm:main0} applied to the graph obtaining from $G$ by adding an isolated vertex $x$ with the non-center pair $(x,x)$. This implies the result. 
\end{proof}

Note that the outcome of a clique cutset in Theorem~\ref{thm:main2} cannot be avoided, as the following example shows. Let $G$ be any $\iskt$-free graph (e.\ g.\ a $C_5$, or a wheel), and let $H$ arise from $G$ by adding $|V(G)|$ disjoint copies of $K_{3,3}$ to $G$ and identifying each vertex of $G$ with a vertex of a different copy of $K_{3,3}$. The resulting graph is $\iskt$-free, not series-parallel, and not bipartite if $G$ is not bipartite, and it contains no vertex of degree at most two.  

Finally we are ready to prove the following. 
\begin{theorem} If $G$ is an $\iskt$-free graph, then $G$ is 3-colorable. 
\end{theorem}
\begin{proof}
The proof is by induction on $|V(G)|$ using Theorem~\ref{v2}. If $G$ is complete bipartite, then $G$ is 2-colorable. If $G$ has a vertex $v$ of degree at most two, then, by induction, $G \setminus v$ is 3-colorable, and hence $G$ is 3-colorable. If $G$ has a clique cutset $C$ such that $(A, B, C)$ is a partition of $V(G)$ with $A$ anticomplete to $B$ and $C$ a clique, then $\chi(G) = \max\sset{\chi(G|(A \cup C)), \chi(G|(B \cup C))}$, and again by induction, $G$ is 3-colorable. 
\end{proof}

\section*{Acknowledgments}

We are thankful to Paul Seymour and Mingxian Zhong for many helpful discussions.


\begin{thebibliography}{50}
\bibitem{duffin} Duffin, Richard J. \emph{Topology of series-parallel networks.} Journal of Mathematical Analysis and Applications 10, no. 2 (1965): 303--318.

\bibitem{le} Khang Le, Ngoc. \emph{Chromatic number of $\iskfour$-free graphs}. arXiv preprint arXiv:1611.04279 (2016). 

\bibitem{leveque} L\'ev\^{e}que, Benjamin, Fr\'ed\'eric Maffray, and Nicolas Trotignon. \emph{On graphs with no induced subdivision of $K_4$.} Journal of Combinatorial Theory, Series B 102, no. 4 (2012): 924--947.

\bibitem{menger} Menger, Karl. \emph{Zur allgemeinen Kurventheorie.} Fundamenta Mathematicae 10, No. 1 (1927): 96--115.

\bibitem{trotignon} Trotignon, Nicolas, and Kristina Vu\v{s}kovi\'c. \emph{On Triangle-Free Graphs That Do Not Contain a Subdivision of the Complete Graph on Four Vertices as an Induced Subgraph.} Journal of Graph Theory (2016).
\end{thebibliography}
\end{document}